\documentclass[12pt]{amsart}
\usepackage{amscd,amssymb,amsthm,amsmath,amssymb,textcomp,supertabular,wasysym,rotating,longtable,enumerate,rotating,mathrsfs,mathtools,copyrightbox}
\usepackage[matrix,arrow,curve]{xy}
\usepackage{pdflscape}

\newtheorem{theorem}[equation]{Theorem}

\newtheorem{proposition}[equation]{Proposition}
\newtheorem{lemma}[equation]{Lemma}
\newtheorem{corollary}[equation]{Corollary}

\newtheorem*{maintheorem*}{Main Theorem}

\theoremstyle{definition}
\newtheorem{example}[equation]{Example}

\theoremstyle{remark}
\newtheorem{remark}[equation]{Remark}

\newcommand{\mumu}{\boldsymbol{\mu}}

\author{Ivan Cheltsov}

\title{Kummer quartic double solids}

\thanks{Throughout this paper, all varieties are assumed to be projective and defined over~$\mathbb{C}$.}

\address{\emph{Ivan Cheltsov}
\newline
\textnormal{The University of Edinburgh,  Edinburgh, Scotland}
\newline
\textnormal{\texttt{I.Cheltsov@ed.ac.uk}}}

\begin{document}

\begin{abstract}
We study equivariant birational geometry of (rational) quartic double solids ramified over (singular) Kummer surfaces.
\end{abstract}

\maketitle

A Kummer quartic surface is an irreducible normal surface in $\mathbb{P}^3$ of degree $4$ that has the~maximal possible number of $16$ singular points,
which are ordinary double~singularities.
Any~such surface is the~Kummer variety of the~Jacobian surface of a smooth genus $2$ curve.
Vice versa, the~Jacobian surface of a smooth genus $2$ curve admits a~natural involution
such that the~quotient surface is a Kummer quartic surface in $\mathbb{P}^3$.

\begin{figure}[h!]
\copyrightbox[r]{\includegraphics[scale=1.0]{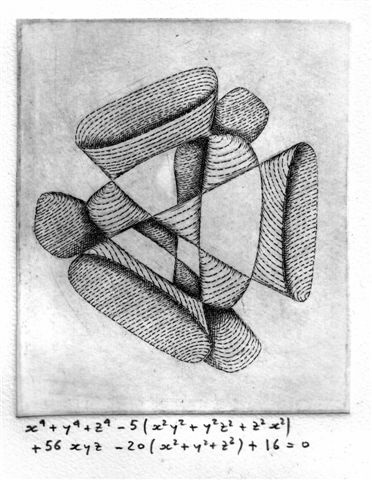}}{\textcopyright Patrice Jeener, La Motte Chalancon, France.}
\caption{A Kummer surface by Patrice Jeener.}
\label{figure:Kummer-surface}
\end{figure}

Let $\mathscr{S}$ be a Kummer surface in $\mathbb{P}^3$, and let $\mathscr{C}$ be the~smooth genus $2$ curve such that
\begin{equation}
\label{equation:Kummer-Jacobian}
\mathscr{S}\cong\mathrm{J}\big(\mathscr{C}\big)/\langle\tau\rangle,
\end{equation}
where $\tau$ is the~involution of the~Jacobian $\mathrm{J}(\mathscr{C})$ that sends a point $P$ to the~point $-P$.
Recall from \cite{Kummer1864,Hudson,Jessop,Nieto,Gonzalez-Dorrego,Dolgachev,Dolgachev2020} that the~surface $\mathscr{S}$ can be given by the~equation
\begin{multline}
\label{equation:Kummer-quartic-surface}
a(x_0^4+x_1^4+x_2^4+x_3^4)+2b(x_0^2x_1^2+x_2^2x_3^2)+\\
+2c(x_0^2x_2^2+x_1^2x_3^2)+2d(x_0^2x_3^2+x_1^2x_2^2)+4ex_0x_1x_2x_3=0
\end{multline}
for some $[a:b:c:d:e]\in\mathbb{P}^4$ such that
\begin{equation}
\label{equation:Segre}
a(a^2+e^2-b^2-c^2-d^2)+2bcd=0.
\end{equation}

Note that the~curve $\mathscr{C}$ is hyperelliptic, and equation \eqref{equation:Segre} defines a cubic threefold in $\mathbb{P}^4$,
which is projectively equivalent to \emph{the Segre cubic threefold}~\cite{Nieto,Dolgachev}.

Using a~formula from the~book \cite{CasselsFlynn} implemented in Magma \cite{Magma},
we can easily extract an~equation of the~surface $\mathscr{S}$ from the~curve~$\mathscr{C}$.
However, the~resulting equation may differ from~\eqref{equation:Kummer-quartic-surface}.
For instance, if $\mathscr{C}$ is the~unique genus $2$ curve such that $\mathrm{Aut}(\mathscr{C})\cong\mumu_2.\mathfrak{S}_4$, then $\mathscr{C}$ is isomorphic to the~curve
$$
\big\{z^2=xy(x^4-y^4)\big\}\subset\mathbb{P}(1,1,3)
$$
where $x$, $y$, $z$ are homogeneous coordinates on $\mathbb{P}(1,1,3)$ of weights $1$, $1$, $2$, respectively.
In~this case, Magma produces the~following Kummer quartic surface:
$$
\big\{x_0^4+2x_0^2x_2x_3-2x_0^2x_2^2+4x_0x_2^2x_2-4x_0x_2x_3^2+x_2^2x_3^2-2x_2x_2^2x_3+x_2^4=0\}\subset\mathbb{P}^3,
$$
which is projectively equivalent to the~surface given by \eqref{equation:Kummer-quartic-surface} with
parameters $a=b=1$, $c=d=-1$, $e=-4$ that do not satisfy~\eqref{equation:Segre}.
But this surface is projectively equivalent to
\begin{equation}
\label{equation:Kummer-surface-S4}
\big\{x_0^4+x_1^4+x_2^4+x_3^4-4ix_0x_1x_2x_3=0\big\}\subset\mathbb{P}^3,
\end{equation}
which is given by \eqref{equation:Kummer-quartic-surface} with parameters $a=1$, $b=c=d=0$, $e=-i$ that do satisfy \eqref{equation:Segre}.
Here, we use the~following Magma code provided to us by Michela Artebani:
\begin{verbatim}
    R<x>:=PolynomialRing(Rationals());
    C:=HyperellipticCurve(x^5-x);
    GroupName(GeometricAutomorphismGroup(C));
    KummerSurfaceScheme(C);
\end{verbatim}

It is not very difficult to recover the~hyperelliptic curve $\mathscr{C}$ from the~quartic surface~$\mathscr{S}$.
Indeed,  $\mathbb{P}^3$ contains $16$ planes $\Pi_1,\ldots,\Pi_{16}$, called \emph{tropes}, such that
$\mathscr{S}\vert_{\Pi_i}=2\mathcal{C}_i$ for each $i$, where $\mathcal{C}_i$ is a smooth conic.
One can show (see, for example, \cite{GriffithsHarris}) that
\begin{itemize}
\item each trope $\Pi_i$ contains exactly six singular points of the~surface $\mathscr{S}$,
\item each singular point of the~surface $\mathscr{S}$ is contained in six planes among $\Pi_1,\ldots,\Pi_{16}$.
\end{itemize}
Moreover, for every conic $\mathcal{C}_i$, there exists a double cover $\mathscr{C}\to\mathcal{C}_i$
which is ramified over the~six points $\mathcal{C}_i\cap\mathrm{Sing}(\mathscr{S})$.
This gives us an algorithm how to recover $\mathscr{C}$ from $\mathscr{S}$.

\begin{example}
\label{example:recover-curve-from-quartic}
Suppose that the~surface $\mathscr{S}$ is given by the~equation \eqref{equation:Kummer-quartic-surface} with
$$
\left\{\aligned
&a=2,\\
&b=-t^2-1,\\
&c=-t^2-1,\\
&d=-t^2-1,\\
&e=t^3+3t,
\endaligned
\right.
$$
where $t\in\mathbb{C}\setminus\{\pm 1,\pm\sqrt{3}i\}$. Then the~surface $\mathscr{S}$ is given by the~following equation:
\begin{equation}
\label{equation:Kummer-quartic-surface-48-50}
x_0^4+x_1^4+x_2^4+x_3^4+2(t^3+3t)x_0x_1x_2x_3=(t^2+1)(x_0^2x_1^2+x_2^2x_3^2+x_0^2x_2^2+x_1^2x_3^2+x_0^2x_3^2+x_1^2x_2^2).
\end{equation}
Its singular locus $\mathrm{Sing}(\mathscr{S})$ consists of the~following $16$ points:
\begin{align*}
[1:1:1:t], [-1:1:-1:t], [-1:-1:1:t], [1:-1:-1:t],\\
[1:1:t:1], [1:-1:t:-1], [-1:-1:t:1], [-1:1:t:-1],\\
[t:1:1:1], [t:-1:1:-1], [t:1:-1:-1], [t:-1:-1:1],\\
[1:t:1:1], [-1:t:-1:1], [1:t:-1:-1], [-1:t:1:-1].
\end{align*}
Moreover, the~tropes $\Pi_1,\ldots,\Pi_{16}$ are listed in the~following table:
\begin{center}
\renewcommand\arraystretch{1.7}
\begin{tabular}{|c||c|}
\hline
$\Pi_1=\big\{x_0+x_1+x_2+tx_3=0\big\}$& $\Pi_2=\big\{x_0-x_1+x_2-tx_3=0\big\}$ \\
\hline
$\Pi_3=\big\{x_0+x_1-x_2-tx_3=0\big\}$& $\Pi_4=\big\{x_0-x_1-x_2+tx_3=0\big\}$\\
\hline
$\Pi_5=\big\{x_0+x_1+tx_2+x_3=0\big\}$& $\Pi_6=\big\{x_0-x_1+tx_2-x_3=0\big\}$ \\
\hline
$\Pi_7=\big\{x_0-tx_2+x_1-x_3=0\big\}$& $\Pi_8=\big\{x_0-x_1-tx_2+x_3=0\big\}$\\
\hline
$\Pi_9=\big\{x_0+tx_1+x_2+x_3=0\big\}$& $\Pi_{10}=\big\{x_0-tx_1-x_2+x_3=0\big\}$ \\
\hline
$\Pi_{11}=\big\{x_0-tx_1+x_2-x_3=0\big\}$& $\Pi_{12}=\big\{x_0+tx_1-x_2-x_3=0\big\}$\\
\hline
$\Pi_{13}=\big\{tx_0+x_1+x_2+x_3=0\big\}$& $\Pi_{14}=\big\{tx_0-x_1-x_2+x_3=0\big\}$ \\
\hline
$\Pi_{15}=\big\{tx_0-x_1+x_2-x_3=0\big\}$& $\Pi_{16}=\big\{tx_0+x_1-x_2-x_3=0\big\}$\\
\hline
\end{tabular}
\end{center}
Then $\mathcal{C}_1$ is the~smooth conic
$$
\big\{x_0+x_1+x_2+tx_3=tx_1x_3+tx_2x_3+x_1^2+x_1x_2+x_2^2-x_3^2=0\big\}\subset\mathbb{P}^3.
$$
This conic contains the~following six singular points of our surface:
\begin{center}
$[1:-1:t:-1]$, $[-1:1:t:-1]$, $[t:-1:1:-1]$, \\
$[t:1:-1:-1]$, $[1:t:-1:-1]$, $[-1:t:1:-1]$.
\end{center}
Projecting from $[t:1:-1:-1]$, we get an isomorphism $\mathcal{C}_1\cong\mathbb{P}^1$
that maps these points~to
\begin{center}
$[t+1:-2]$, $[1:0]$, $[-1:1]$, $[1-t:1+t]$, $[0:1]$, $[t-1:2]$.
\end{center}
Therefore, the~hyperelliptic curve $\mathscr{C}$ is isomorphic to the~curve
$$
\big\{z^2=xy\big(x-y\big)\big((t-1)x+2y\big)\big(2x-(t+1)y\big)\big((t+1)x-(t-1)y\big)\big\}\subset\mathbb{P}(1,1,3).
$$
In particular, it follows from \cite{CGLR} or Magma computations that
$$
\mathrm{Aut}\big(\mathscr{C}\big)\cong\left\{\aligned
&\mumu_2.\mathfrak{S}_{4}\ \text{if}\ t\in\{0,\pm i, 1\pm 2i, -1\pm 2i\},\\
&\mumu_2.\mathrm{D}_{12}\ \text{if}\ t\in\{0,\pm 3\},\\
&\mumu_2\times\mathfrak{S}_3\ \text{if $t$ is general}.
\endaligned
\right.
$$
For instance, to identify $\mathrm{Aut}(\mathscr{C})$ in the~case when $t=i$, one can use the~following script:
\begin{verbatim}
    K:=CyclotomicField(4);
    R<x>:=PolynomialRing(K);
    i:=Roots(x^2+1,K)[1,1];
    t:=i;
    f:=x*(x-1)*((t-1)*x+2)*(2*x-(t+1))*((t+1)*x-(t-1));
    C:=HyperellipticCurve(f);
    GroupName(GeometricAutomorphismGroup(C));
\end{verbatim}
In this example, we assume that $t\not\in\{\pm 1,\pm\sqrt{3}i\}$, because
\begin{itemize}
\item if $t=\pm 1$ or $t=\infty$, then the~equation \eqref{equation:Kummer-quartic-surface-48-50} defines a union of $4$ planes,
\item if $t=\pm\sqrt{3}i$, the~equation \eqref{equation:Kummer-quartic-surface-48-50} defines a double quadric.
\end{itemize}
These are semistable degenerations with minimal $\mathrm{PGL}_4(\mathbb{C})$-orbits \cite[Theorem~2.4]{Shah}.
\end{example}

Let $\mathrm{Aut}(\mathbb{P}^3,\mathscr{S})$ be the~subgroup in $\mathrm{PGL}_4(\mathbb{C})$ consisting
of projective transformations that leave $\mathscr{S}$ invariant.
Then $\mathrm{Aut}(\mathbb{P}^3,\mathscr{S})$ contains a subgroup $\mathbb{H}\cong\mumu_2^4$  generated~by
$$\hspace*{-0.3cm}
A_1=\begin{pmatrix}
-1 & 0 & 0 & 0\\
0 & 1 & 0 & 0\\
0 & 0 & -1 & 0\\
0 & 0 & 0 & 1
\end{pmatrix},
A_2=\begin{pmatrix}
-1 & 0 & 0 & 0\\
0 & -1 & 0 & 0\\
0 & 0 & 1 & 0\\
0 & 0 & 0 & 1
\end{pmatrix},
A_3=\begin{pmatrix}
0 & 1 & 0 & 0\\
1 & 0 & 0 & 0\\
0 & 0 & 0 & 1\\
0 & 0 & 1 & 0
\end{pmatrix},
A_4=\begin{pmatrix}
0 & 0 & 0 & 1\\
0 & 0 & 1 & 0\\
0 & 1 & 0 & 0\\
1 & 0 & 0 & 0
\end{pmatrix}.
$$
The~action of this group on $\mathscr{S}$ is induced by translations of  $\mathrm{J}(\mathscr{C})$ by~\mbox{two-torsion} points~\cite{Klein}, so $\mathrm{Sing}(\mathscr{S})$ is an $\mathbb{H}$-orbit.
Similarly, we see that $\mathbb{H}$ acts transitively~on the~set
\begin{equation}
\label{equation:16-planes}
\Big\{\Pi_1,\Pi_2,\Pi_3,\Pi_4,\Pi_5,\Pi_6,\Pi_7,\Pi_8,\Pi_9,\Pi_{10},\Pi_{11},\Pi_{12},\Pi_{13},\Pi_{14},\Pi_{15},\Pi_{16}\Big\}.
\end{equation}

If $\mathscr{S}$ is general, then $\mathrm{Aut}(\mathbb{P}^3,\mathscr{S})=\mathbb{H}$,
and $\mathrm{Aut}(\mathscr{C})$ is generated by the~hyperelliptic involution \cite{Keum1997,Kondo}.
However, if $\mathscr{S}$ is special, then $\mathrm{Aut}(\mathbb{P}^3,\mathscr{S})$ can be larger than $\mathbb{H}$.

\begin{example}
\label{example:48-50}
Let us use assumptions and notations of Example~\ref{example:recover-curve-from-quartic}.
For~$t\in\mathbb{C}\setminus\{\pm 1,\pm\sqrt{3}i\}$, the~group $\mathrm{Aut}(\mathbb{P}^3,\mathscr{S})$ contains the~subgroup isomorphic to $\mumu_2^4\rtimes\mathfrak{S}_3$
generated by
$$
A_1,A_2,A_3,A_4,
\begin{pmatrix}
0 & 0 & 0 & 1\\
1 & 0 & 0 & 0\\
0 & 1 & 0 & 0\\
0 & 0 & 1 & 0
\end{pmatrix},
\begin{pmatrix}
0 & 1 & 0 & 0\\
1 & 0 & 0 & 0\\
0 & 0 & 1 & 0\\
0 & 0 & 0 & 1
\end{pmatrix}.
$$
In fact, this is the~whole group $\mathrm{Aut}(\mathbb{P}^3,\mathscr{S})$ if $t$ is general.
On the~other hand, if $t=0$, then it follows from \cite{Edge,Catanese} that $\mathrm{Aut}(\mathbb{P}^3,\mathscr{S})\cong\mumu_2^4\rtimes\mathrm{D}_{12}$,
and this group is generated by
$$
A_1,A_2,A_3,A_4,
\begin{pmatrix}
1 & 0 & 0 & 0\\
0 & 1 & 0 & 0\\
0 & 0 & -1 & 0\\
0 & 0 & 0 & 1
\end{pmatrix},
\begin{pmatrix}
0 & 0 & 0 & 1\\
1 & 0 & 0 & 0\\
0 & 1 & 0 & 0\\
0 & 0 & 1 & 0
\end{pmatrix},
\begin{pmatrix}
0 & 1 & 0 & 0\\
1 & 0 & 0 & 0\\
0 & 0 & 1 & 0\\
0 & 0 & 0 & 1
\end{pmatrix}.
$$
If $t=\pm i$, then $\mathscr{S}$ is the~surface \eqref{equation:Kummer-surface-S4},
and $\mathrm{Aut}(\mathbb{P}^3,\mathscr{S})\cong\mumu_2^4\rtimes\mathfrak{S}_{4}$ is generated by
$$
A_1,A_2,A_3,A_4,
\begin{pmatrix}
0 & 0 & 0 & 1\\
1 & 0 & 0 & 0\\
0 & 1 & 0 & 0\\
0 & 0 & 1 & 0
\end{pmatrix},
\begin{pmatrix}
0 & 1 & 0 & 0\\
1 & 0 & 0 & 0\\
0 & 0 & 1 & 0\\
0 & 0 & 0 & 1
\end{pmatrix},
\begin{pmatrix}
i & 0 & 0 & 0\\
0 & i & 0 & 0\\
0 & 0 & -1 & 0\\
0 & 0 & 0 & 1
\end{pmatrix}.
$$
\end{example}

\begin{example}
\label{example:Z5}
Suppose that $\mathscr{S}$ is given by the~equation \eqref{equation:Kummer-quartic-surface} with
$$
\left\{\aligned
&a=2\zeta_5^3+2\zeta_5^2+6\zeta_5-1,\\
&b=4\zeta_5^3+4\zeta_5^2-10\zeta_5+9,\\
&c=-6\zeta_5^3-6\zeta_5^2+4\zeta_5+3,\\
&d=11,\\
&e=-20\zeta_5^3+24\zeta_5^2-16\zeta_5+10.
\endaligned
\right.
$$
Then $\mathrm{Aut}(\mathscr{C})\cong\mumu_2\times\mumu_5$ and $\mathrm{Aut}(\mathbb{P}^3,\mathscr{S})\cong\mumu_2^4\rtimes\mumu_5$,
which is generated by
$$
A_1,A_2,A_3,A_4,
\begin{pmatrix}
-i & 0 & 0 & i\\
0 & 1 & 1 & 0\\
1 & 0 & 0 & 1\\
0 & -i & i & 0
\end{pmatrix}.
$$
\end{example}

Looking at Examples~\ref{example:recover-curve-from-quartic}, \ref{example:48-50} and \ref{example:Z5},
we can see a relation between $\mathrm{Aut}(\mathbb{P}^3,\mathscr{S})$ and $\mathrm{Aut}(\mathscr{C})$.
In fact, this relation holds for all Kummer surfaces in $\mathbb{P}^3$ by the~following well-known~result,
about which we learned from Igor Dolgachev.

\begin{lemma}
\label{lemma:Dolgachev}
Let $\iota\in\mathrm{Aut}(\mathscr{C})$ be the~hyperelliptic involution of the~curve $\mathscr{C}$. Then
$$
\mathrm{Aut}\big(\mathbb{P}^3,\mathscr{S}\big)\cong\mumu_2^4\rtimes\big(\mathrm{Aut}(\mathscr{C})/\langle\iota\rangle\big).
$$
\end{lemma}

\begin{proof}
Let us identify $\mathscr{C}$ with the~theta divisor in $\mathrm{J}(\mathscr{C})$ via the~Abel--Jacobi~map whose base point is one of the~fixed points of the~involution $\iota$ (one of the~six Weierstrass points).
Then the~linear system $|2\mathscr{C}|$ gives a morphism $\mathrm{J}(\mathscr{C})\to\mathbb{P}^3$ whose image is the~surface~$\mathscr{S}$.
Taking the~Stein factorization of the~morphism $\mathrm{J}(\mathscr{C})\to\mathscr{S}$,
we get the~isomorphism \eqref{equation:Kummer-Jacobian}.

On the~other hand, elements in $\mathrm{Aut}(\mathscr{C})$ give automorphisms in $\mathrm{Aut}(\mathrm{J}(\mathscr{C}))$  that leave
the~linear system $|2\mathscr{C}|$ invariant. This gives us a homomorphism
$\mathrm{Aut}(\mathscr{C})\to\mathrm{Aut}(\mathbb{P}^3,\mathscr{S})$,
whose kernel is the~hyperelliptic involution $\iota$, since $\iota$ induces the~involution $\tau\in\mathrm{Aut}(\mathrm{J}(\mathscr{C}))$.

The image of the~group $\mathrm{Aut}(\mathscr{C})$ in $\mathrm{Aut}(\mathbb{P}^3,\mathscr{S})$ normalizes the~subgroup $\mathbb{H}$,
because elements in $\mathbb{H}$ are induced by the~translations of the~Jacobian $\mathrm{J}(\mathscr{C})$ by two-torsion points.
This gives a~monomorphism
$\vartheta\colon \mumu_2^4\rtimes(\mathrm{Aut}(\mathscr{C})/\langle\iota\rangle)\to\mathrm{Aut}(\mathbb{P}^3,\mathscr{S})$.

We claim that $\vartheta$ is an epimorphism.
Indeed, the~action of an element $g\in\mathrm{Aut}(\mathbb{P}^3,\mathscr{S})$~on the~surface $\mathscr{S}$
lifts to its its action on the~Jacobian $\mathrm{J}(\mathscr{C})$ that leaves $[2\mathscr{C}]$ invariant,
so~composing $g$ with some $h\in\mathbb{H}$, we obtain an~element $g\circ h$ that preserves the~class~$[\mathscr{C}]$.
Thus, since $[\mathscr{C}]$ is a principal polarization, the~composition $g\circ h$ preserves $\mathscr{C}$,
and it acts faithfully on $\mathscr{C}$, since $\mathscr{C}$ generates $\mathrm{J}(\mathscr{C})$.
This gives $g\circ h\in\mathrm{im}(\vartheta)$, so $\vartheta$ is surjective.
\end{proof}

Since $\mathrm{Aut}(\mathscr{C})$ is isomorphic to a group among $\mumu_2$, $\mumu_2^2$, $\mathrm{D}_{8}$, $\mathrm{D}_{12}$, $\mumu_2.\mathrm{D}_{12}$, $\mumu_2.\mathfrak{S}_{4}$,~$\mumu_2\times\mumu_5$,
we conclude that $\mathrm{Aut}(\mathbb{P}^3,\mathscr{S})$ is isomorphic to one of the~following groups:
\begin{center}
$\mumu_2^4$, $\mumu_2^4\rtimes\mumu_2$,  $\mumu_2^4\rtimes\mumu_2^2$, $\mumu_2^4\rtimes\mathfrak{S}_3$,
$\mumu_2^4\rtimes\mathrm{D}_{12}$, $\mumu_2^4\rtimes\mathfrak{S}_{4}$, $\mumu_2^4\rtimes\mumu_5$.
\end{center}
Note that the~group $\mathrm{Aut}(\mathscr{S})$ is always larger that $\mathrm{Aut}(\mathbb{P}^3,\mathscr{S})$ \cite{Keum1997,Kondo}.

\begin{remark}[{\cite{Blichfeldt1917,Nieto,Gonzalez-Dorrego}}]
\label{remark:S6}
Let $\mathfrak{N}$ be the~normalizer of the~subgroup $\mathbb{H}$ in the~group $\mathrm{PGL}_4(\mathbb{C})$.
Then $\mathrm{Aut}(\mathbb{P}^3,\mathscr{C})\subset\mathfrak{N}$,
and there exists an exact sequence $1\longrightarrow \mathbb{H}\longrightarrow \mathfrak{N}\longrightarrow \mathfrak{S}_{6}\longrightarrow 1$,
which can be described as follows. Let
$$
B_1=\begin{pmatrix}
    i & 0 & 0& 0\\
    0 & i& 0& 0\\
    0 & 0 & 1& 0\\
    0 & 0 & 0& 1
 \end{pmatrix}\ \text{and}\ B_2=\begin{pmatrix}
    -i& 0& 0& i\\
    0 & 1& 1& 0\\
    1 & 0& 0& 1\\
    0 & -i& i& 0
 \end{pmatrix}.
$$
Then $\langle B_1,B_2\rangle\in\mathfrak{N}$. Since $B_1^2\in\mathbb{H}$, $B_2^5=(B_1B_2)^6=[B_1,B_2]^3=\mathrm{Id}_{\mathbb{P}^3}$, $[B_1,B_2B_1B_2]^2\in\mathbb{H}$,
the images of $B_1$ and $B_2$ in the~quotient $\mathfrak{N}/\mathbb{H}$ generate the~whole group $\mathfrak{N}/\mathbb{H}\cong\mathfrak{S}_6$. Set
\begin{align*}
S_1&=\big\{x_0^4+x_1^4+x_2^4+x_3^4-6\big(x_0^2x_1^2+x_2^2x_3^2\big)-6\big(x_0^2x_2^2+x_1^2x_3^2\big)-6\big(x_0^2x_3^2+x_1^2x_2^2\big)=0\big\},\\
S_2&=\big\{x_0^4+x_1^4+x_2^4+x_3^4-6\big(x_0^2x_1^2+x_2^2x_3^2\big)+6\big(x_0^2x_2^2+x_1^2x_3^2\big)+6\big(x_0^2x_3^2+x_1^2x_2^2\big)=0\big\},\\
S_3&=\big\{x_0^4+x_1^4+x_2^4+x_3^4+6\big(x_0^2x_1^2+x_2^2x_3^2\big)-6\big(x_0^2x_2^2+x_1^2x_3^2\big)+6\big(x_0^2x_3^2+x_1^2x_2^2\big)=0\big\},\\
S_4&=\big\{x_0^4+x_1^4+x_2^4+x_3^4+6\big(x_0^2x_1^2+x_2^2x_3^2\big)+6\big(x_0^2x_2^2+x_1^2x_3^2\big)-6\big(x_0^2x_3^2+x_1^2x_2^2\big)=0\big\},\\
S_5&=\big\{x_0^4+x_1^4+x_2^4+x_3^4-12x_0x_1x_2x_3=0\big\},\\
S_6&=\big\{x_0^4+x_1^4+x_2^4+x_3^4+12x_0x_1x_2x_3=0\big\}.
\end{align*}
Then $S_1$, $S_2$, $S_3$, $S_4$, $S_5$, $S_6$ are $\mathbb{H}$-invariant surfaces,
and the~quotient $\mathfrak{N}/\mathbb{H}$ permutes~them.
For instance, the~transformation $B_1$ acts on the~set $\{S_1,S_2,S_3,S_4,S_5,S_6\}$~as~\mbox{$(1\,2)(3\,4)(5\,6)$},
and $B_2$ acts as the~permutation~\mbox{$(1\,2\,6\,3\,5)$}. This gives an explicit isomorphism $\mathfrak{N}/\mathbb{H}\cong\mathfrak{S}_6$.
\end{remark}

\begin{remark}
\label{remark:Dolgachev}
The quotient $\mathrm{Aut}(\mathbb{P}^3,\mathscr{S})/\mathbb{H}$ naturally acts on the~threefold \eqref{equation:Segre}
fixing the~point $[a:b:c:d:e]$ that corresponds to $\mathscr{S}$.
Projecting the~threefold from~this~point, we obtain a (rational) double cover of $\mathbb{P}^3$ that is branched along the~surface $\mathscr{S}$.
\end{remark}

Let $\pi\colon X\to\mathbb{P}^3$ be the~double cover branched along the~surface $\mathscr{S}$. Set~$H=\pi^*(\mathcal{O}_{\mathbb{P}^3}(1))$.
Then $\mathrm{Pic}(X)=\mathbb{Z}[H]$, $H^3=2$ and $-K_{X}\sim 2H$,
so $X$ is a del Pezzo threefold of degree $2$, which has $16$ ordinary double points.
We say that $X$ is \emph{a~Kummer~quartic~double~solid}~\cite{Varley}.

The~threefold $X$ is a hypersurface in $\mathbb{P}(1,1,1,1,2)$ given by
\begin{multline}
\label{equation:double-quartic}
w^2=a(x_0^4+x_1^4+x_2^4+x_3^4)+2b(x_0^2x_1^2+x_2^2x_3^2)+\\
+2c(x_0^2x_2^2+x_1^2x_3^2)+2d(x_0^2x_3^2+x_1^2x_2^2)+4ex_0x_1x_2x_3,
\end{multline}
where we consider $x_0$, $x_1$, $x_2$, $x_3$ as homogeneous coordinates on $\mathbb{P}(1,1,1,1,2)$ of weight~$1$,
and $w$ is a homogeneous coordinate on $\mathbb{P}(1,1,1,1,2)$ of weight $2$.

It is well-known that the~threefold $X$ is rational \cite{Semple-Roth,Varley,Prokhorov2013,CheltsovPrzyjalkowskiShramov}, see also Remark~\ref{remark:Dolgachev}.
Let us describe one birational map $\mathbb{P}^3\dasharrow X$ following \cite{Semple-Roth,Prokhorov2013}.
To do this, fix a~twisted cubic curve $C_3\subset\mathbb{P}^3$.
Then $C_3$ contains six distinct points $P_1,\ldots,P_6$ such that there exists a~commutative diagram
\begin{equation}
\label{equation:Prokhorov}
\xymatrix{
\widehat{X}\ar@{->}[d]_{\eta}\ar@{->}[rr]^{\varphi}&&X\ar@{->}[d]^{\pi}\\%
\mathbb{P}^3\ar@{-->}[rr]^{\chi}&&\mathbb{P}^3}
\end{equation}
where $\eta$ is a blow up of the points $P_1,\ldots,P_6$,
the morphism $\varphi$ is a contraction of the~proper transform of the~curve $C_3$ and proper transforms of $15$ lines in $\mathbb{P}^3$ that
pass through two points among $P_1,\ldots,P_6$, and $\chi$ is a rational map given by the~linear system of quadric surfaces
that pass through the points $P_1,\ldots,P_6$.

\begin{corollary}[{\cite{Clemens,Endrass}}]
\label{corollary:class-group}
One has $\mathrm{Cl}(X)\cong\mathbb{Z}^7$.
\end{corollary}

\begin{remark}
\label{remark:Weddle}
The vertices of the~quadric cones in $\mathbb{P}^3$ that pass through $P_1,\ldots,P_6$
span an~irreducible singular quartic surface $\mathfrak{S}$ which is known as \emph{the~Weddle surface} \cite{Hudson,Varley}.
This surface is singular at the points $P_1,\ldots,P_6$, and $\chi$ induces a birational map $\mathfrak{S}\dasharrow\mathscr{S}$.
On the~other hand, the~double cover of $\mathbb{P}^3$  branched along $\mathfrak{S}$ is irrational \cite{Varley,CheltsovPrzyjalkowskiShramov}.
\end{remark}

Let $\sigma\in\mathrm{Aut}(X)$ be the~Galois involution of the~double cover $\pi$.
Then $\sigma$ is contained in the~center of the~group $\mathrm{Aut}(X)$.
Moreover, since $\pi$ is $\mathrm{Aut}(X)$-equivariant, it induces
a~homomorphism $\upsilon\colon\mathrm{Aut}(X)\to\mathrm{Aut}(\mathbb{P}^3,\mathscr{S})$ with $\mathrm{ker}(\upsilon)=\langle\sigma\rangle$,
so we have exact sequence
$$
\xymatrix{
1\ar@{->}[rr]&&\langle\sigma\rangle\ar@{->}[rr]&&\mathrm{Aut}(X)\ar@{->}[rr]^{\upsilon}&&\mathrm{Aut}\big(\mathbb{P}^3,\mathscr{S}\big)\ar@{->}[rr]&&1.}
$$

The main result of this paper is the~following theorem (cf. \cite{Avilov2019,Avilov,CheltsovPrzyjalkowskiShramov2016}).

\begin{theorem}
\label{theorem:main}
Let $G$ be any subgroup in $\mathrm{Aut}(X)$ such that $\mathrm{Cl}^G(X)\cong\mathbb{Z}$ and $\mathbb{H}\subseteq\upsilon(G)$.
Then the~Fano threefold $X$ is $G$-birationally super-rigid.
\end{theorem}

\begin{corollary}
\label{corollary:main}
Let $G$ be any subgroup in $\mathrm{Aut}(X)$ such that $G$ contains $\sigma$ and $\mathbb{H}\subseteq\upsilon(G)$.
Then $X$ is $G$-birationally super-rigid.
\end{corollary}

The~condition  $\mathrm{Cl}^G(X)\cong\mathbb{Z}$ in Theorem~\ref{theorem:main} simply means that $X$ is a~$G$-Mori fibre space,
which is required by the~definition of $G$-birational super-rigidity (see \cite[Definition~3.1.1]{CheltsovShramov}).
The condition $\mathbb{H}\subseteq\upsilon(G)$ does not imply that $\mathrm{Cl}^G(X)\cong\mathbb{Z}$,
see Examples~\ref{example:Heisenberg-class-group-Z-Z} and \ref{example:80-49}  below.
The following example shows that we cannot remove the~condition $\mathbb{H}\subseteq\upsilon(G)$.

\begin{example}
\label{example:Z2-non-rigid}
Observe that $\mathrm{Cl}^{\langle\sigma\rangle}(X)\cong\mathbb{Z}$.
Let $S_1$ and $S_2$ be two general surfaces in $|H|$, and let $C=S_1\cap S_2$.
Then~$C$ is a smooth irreducible $\langle\sigma\rangle$-invariant curve, $\pi(C)$ is a line,
and there exists $\langle\sigma\rangle$-commutative diagram
$$
\xymatrix{
&V\ar@{->}[ld]_{\alpha}\ar@{->}[rd]^{\beta}&\\%
X\ar@{-->}[rr]&&\mathbb{P}^1}
$$
where $\alpha$ is the~blow up of the~curve $C$,
the dashed arrow $\dasharrow$ is given by the~pencil generated by the~surfaces $S_1$ and $S_2$,
and $\beta$ is a fibration into del Pezzo surfaces of~degree~$2$.
Therefore, the~threefold $X$ is not $\langle\sigma\rangle$-birationally rigid.
\end{example}

Let $G$ be a  subgroup in $\mathrm{Aut}(X)$ such that $\upsilon(G)$ contains $\mathbb{H}$.
Before proving Theorem~\ref{theorem:main}, let us explain how to check the~condition $\mathrm{Cl}^G(X)\cong\mathbb{Z}$.
For a homomorphism $\rho\colon\mathbb{H}\to\mumu_2$, consider the~action of the~group $\mathbb{H}$ on the~threefold $X$ given by
\begin{align*}
A_1\colon [x_0:x_1:x_2:x_3:w]&\mapsto[-x_0:x_1:-x_2:x_3:\rho(A_1)w],\\
A_2\colon [x_0:x_1:x_2:x_3:w]&\mapsto[-x_0:x_1:-x_2:x_3:\rho(A_2)w],\\
A_3\colon [x_0:x_1:x_2:x_3:w]&\mapsto[x_1:x_2:x_3:x_2:\rho(A_3)w],\\
A_4\colon [x_0:x_1:x_2:x_3:w]&\mapsto[x_3:x_2:x_1:x_0:\rho(A_4)w].
\end{align*}
This gives a lift of the~subgroup $\mathbb{H}$ to $\mathrm{Aut}(X)$.
Let $\mathbb{H}^\rho$ be the~resulting subgroup~in~$\mathrm{Aut}(X)$.
Since $\mathbb{H}\subset\upsilon(G)$, we may assume that $\mathbb{H}^\rho\subset G$.
If $\rho$ is trivial, we let $\mathbb{H}=\mathbb{H}^\rho$ for simplicity.

For every plane $\Pi_i$, one has $\pi^*(\Pi_i)=\Pi_i^++\Pi_i^-$,
where $\Pi_i^+$ and $\Pi_i^-$ are two irreducible surfaces such that $\Pi_i^+\ne\Pi_i^-$ and $\sigma(\Pi_i^+)=\Pi_i^-$.
Note that we do not have a canonical way to distinguish between the~surfaces $\Pi_i^+$ and~$\Pi_i^-$. Namely, if  $\pi^*(\Pi_i)$ is given by
$$
\left\{\aligned
&h_i(x_0,x_1,x_2,x_3)=0,\\
&w^2=g_i^2(x_0,x_1,x_2,x_3),
\endaligned
\right.
$$
where $h_i$ is a linear polynomial such that $\Pi_i=\{h_i=0\}\subset\mathbb{P}^3$,
and $g_i$ is a quadratic polynomial such that the~conic $\mathcal{C}_i$ is given by $h_i=g_i=0$, then
$$
\Pi_i^\pm=\big\{w\pm g_i(x_0,x_1,x_2,x_3)=h_i(x_0,x_1,x_2,x_3)=0\big\}\subset\mathbb{P}(1,1,1,1,2).
$$
But the~choice of $\pm$ here is not uniquely defined, because we can always swap $g_i$ with $-g_i$.

On the~other hand, since $\mathbb{H}$ acts transitively on the~set \eqref{equation:16-planes}, the~set
$$
\Big\{\Pi_1^+,\Pi_1^-,\Pi_2^+,\Pi_2^-,\Pi_3^+,\Pi_3^-\ldots,\Pi_{14}^+,\Pi_{14}^-,\Pi_{15}^+,\Pi_{15}^-,\Pi_{16}^+,\Pi_{16}^-\Big\}
$$
splits into two $\mathbb{H}^\rho$-orbits consisting of $16$ surfaces such that each of them contains
exactly one surface among $\Pi_i^+$ and $\Pi_i^-$ for every $i$.
Hence, we may assume that these $\mathbb{H}^\rho$-orbits~are
$$
\Big\{\Pi_1^+,\Pi_2^+,\Pi_3^+,\Pi_4^+,\Pi_5^+,\Pi_6^+,\Pi_7^+,\Pi_8^+,\Pi_9^+,\Pi_{10}^+,\Pi_{11}^+,\Pi_{12}^+,\Pi_{13}^+,\Pi_{14}^+,\Pi_{15}^+,\Pi_{16}^+\Big\}
$$
and
$$
\Big\{\Pi_1^-,\Pi_2^-,\Pi_3^-,\Pi_4^-,\Pi_5^-,\Pi_6^-,\Pi_7^-,\Pi_8^-,\Pi_9^-,\Pi_{10}^-,\Pi_{11}^-,\Pi_{12}^-,\Pi_{13}^-,\Pi_{14}^-,\Pi_{15}^-,\Pi_{16}^-\Big\}.
$$
\begin{corollary}
\label{corollary:class-group-planes}
The surfaces $\Pi_1^+,\Pi_1^-,\ldots,\Pi_{16}^+,\Pi_{16}^-$ generate the~group $\mathrm{Cl}(X)$.
\end{corollary}

\begin{corollary}
\label{corollary:class-group-planes-H-invariant-part}
Either $\mathrm{Cl}^{\mathbb{H}^\rho}(X)\cong\mathbb{Z}$ or $\mathrm{Cl}^{\mathbb{H}^\rho}(X)\cong\mathbb{Z}^2$.
\end{corollary}

The~surfaces $\Pi_1^+,\Pi_1^-,\ldots,\Pi_{16}^+,\Pi_{16}^-$ are called \emph{planes} \cite{Semple-Roth}.
They are not Cartier~divisors.
To describe their strict transforms on the~threefold $\widehat{X}$ from the~commutative~diagram~\eqref{equation:Prokhorov},
let us introduce the following notations:
\begin{itemize}
\item[(a)] let $E_i$ be the $\eta$-exceptional surfaces such that $\eta(E_i)=P_i$, where $1\leqslant i\leqslant 6$;
\item[(b)] let $\widehat{\Pi}_{i,j,k}$ be the strict transform on the threefold $\widehat{X}$ of the plane in $\mathbb{P}^3$ that passes through
the three distinct points $P_i$, $P_j$, $P_k$, where $1\leqslant i<j<k\leqslant 6$;
\item[(c)] let $\widehat{Q}_i$ be the strict transform of the irreducible quadric cone in $\mathbb{P}^3$
that contains all points $P_1,\ldots,P_6$ and is singular at the point $P_i$, where $1\leqslant i\leqslant 6$.
\end{itemize}
Then strict transforms of the~surfaces $\Pi_1^+,\Pi_1^-,\ldots,\Pi_{16}^+,\Pi_{16}^-$
on the threefold $\widehat{X}$ are
$$
E_1,E_2,\ldots,E_6,\widehat{\Pi}_{1,2,3},\widehat{\Pi}_{1,2,4},\ldots,\widehat{\Pi}_{3,5,6},\widehat{\Pi}_{4,5,6},\widehat{Q}_1,\widehat{Q}_2,\ldots\widehat{Q}_6.
$$
The involution $\sigma$ acts birationally on $\widehat{X}$ as a composition of flops of $\varphi$-contracted curves,
it~swaps $E_i\leftrightarrow\widehat{Q}_i$, and it swaps $\widehat{\Pi}_{i,j,k}\leftrightarrow\widehat{\Pi}_{r,s,t}$ such that~\mbox{$\{i,j,k,r,s,t\}=\{1,2,3,4,5,6\}$}.
For a non-trivial element $g\in\langle\tau,\mathbb{H}^\rho\rangle$,
the involution $\varphi^{-1}\circ g\circ\varphi$ is a composition of flops,
and either $\eta\circ\varphi^{-1}\circ g\circ\varphi\circ\eta^{-1}$ or $\eta\circ\varphi^{-1}\circ g\circ\sigma\circ\varphi\circ\eta^{-1}$
is a Cremona involution whose fundamental points are four points among $P_1,\ldots,P_6$ that swaps the remaining two points.

Now, let us give a criterion for $\mathrm{Cl}^G(X)\cong\mathbb{Z}$. To do this, we set
$$
\Pi^\pm=\sum_{i=1}^{16}\Pi_i^\pm.
$$
Then $\Pi^+$ and $\Pi^-$ are $\mathbb{H}^\rho$-invariant divisors, $\sigma(\Pi^+)=\Pi^-$ and $\Pi^++\Pi^-\sim 16H$.

\begin{lemma}
\label{lemma:class-group-planes}
One has $\mathrm{Cl}^{G}(X)\cong\mathbb{Z}$ is at least one of the~following conditions is satisfied:
\begin{itemize}
\item[(i)] the~group $G$ swaps $\Pi^+$ and $\Pi^-$;
\item[(ii)] the~divisor $\Pi^+$ is Cartier;
\item[(iii)] the~divisor $\Pi^-$ is Cartier;
\item[(iv)] the~surfaces $\Pi_1^+,\ldots,\Pi_{16}^+$ generate the~group $\mathrm{Cl}(X)$;
\item[(v)] the~surfaces $\Pi_1^-,\ldots,\Pi_{16}^-$ generate the~group $\mathrm{Cl}(X)$.
\end{itemize}
\end{lemma}

\begin{proof}
The assertion follows from Corollary~\ref{corollary:class-group-planes-H-invariant-part}, since we assume that $\mathbb{H}^\rho\subset G$.
\end{proof}

This lemma is easy to apply if we fix $\mathscr{S}$ and the~group $G\subset\mathrm{Aut}(X)$ such that $\mathbb{H}\subset\upsilon(G)$.
For instance, to check whether the~surfaces $\Pi_1^+,\ldots,\Pi_{16}^+$ generate the~group $\mathrm{Cl}(X)$ or not,
we can use the~fact that $\mathrm{Cl}(X)\cong\mathbb{Z}^7$ is naturally equipped with an intersection form \cite{Prokhorov2013}.
Namely, fix a smooth del Pezzo surface $S\in |H|$, and let
$$
D_1\bullet D_2=D_1\big\vert_{S}\cdot D_2\big\vert_{S}\in\mathbb{Z}
$$
for any two Weil divisors $D_1$ and $D_2$ in $\mathrm{Cl}(X)$.
Then
$$
\Pi_i^\pm\bullet \Pi_j^\pm=\left\{\aligned
&0\ \text{if $i\ne j$ and $\Pi_i^\pm\cap \Pi_j^\pm$ does not contain curves},\\
&1\ \text{if $i\ne j$ and $\Pi_i^\pm\cap \Pi_j^\pm$ contains a curve},\\
&-1\ \text{if $i=j$ and $\Pi_i^\pm=\Pi_j^\pm$},\\
&2\ \text{if $i=j$ and $\Pi_i^\pm\ne \Pi_j^\pm$},\\
\endaligned
\right.
$$
where two $\pm$ in $\Pi_i^\pm$ and $\Pi_j^\pm$ are independent.

\begin{remark}
\label{remark:Weyl-group}
Let $\Lambda$ be the~sublattice in $\mathrm{Cl}(X)$ consisting of divisors $D$ such that $D\bullet H=0$.
Then $\Lambda$ is isomorphic to a root lattice of type $\mathrm{D}_6$ by \cite[Theorem~1.7]{Prokhorov2013},
and the~natural homomorphism $\mathrm{Aut}(X)\to\mathrm{Aut}(\Lambda)$ is injective \cite{Prokhorov2013},
where $\mathrm{Aut}\big(\Lambda\big)\cong(\mumu_2^5\rtimes\mathfrak{S}_6)\rtimes\mumu_2$.
\end{remark}

Applying Lemma~\ref{lemma:class-group-planes}, we get

\begin{corollary}
\label{corollary:class-group-simple-criterion}
If $\mathrm{rank}(\Pi_i^+\bullet \Pi_j^+)=7$ or $\mathrm{rank}(\Pi_i^-\bullet \Pi_j^-)=7$, then $\mathrm{Cl}^{\mathbb{H}^\rho}(X)\cong\mathbb{Z}$.
\end{corollary}

Let us show how to apply Corollary~\ref{corollary:class-group-simple-criterion}.

\begin{example}
\label{example:class-group-planes-48-50}
Let us use assumptions and notations of Example~\ref{example:recover-curve-from-quartic}.
Suppose, in addition, that $\rho\colon\mathbb{H}\to\mumu_2$ is the~trivial homomorphism.
Therefore, we have $\mathbb{H}^\rho=\mathbb{H}$.
Set $t=\frac{2s}{s^2 + 1}$.
Observe that $\pi^*(\Pi_1)$ is given in $\mathbb{P}(1,1,1,1,2)$ by the~following equations:
$$
\left\{\aligned
&x_0+x_1+x_2+\frac{2s}{s^2+1}x_3=0,\\
&w^2=\frac{(s^2-1)^2}{(s^2+1)^4}\big((s^2+1)x_1^2+(s^2+1)x_1x_2+2sx_1x_3+(s^2+1)x_2^2+2sx_2x_3-(s^2+1)x_3^2\big)^2.
\endaligned
\right.
$$
Thus, without loss of generality, we may assume that the~surface $\Pi_1^+$ is given by
$$
\left\{\aligned
&x_0+x_1+x_2+\frac{2s}{s^2+1}x_3=0,\\
&w=\frac{s^2-1}{(s^2+1)^2}\big((s^2+1)x_1^2+(s^2+1)x_1x_2+2sx_1x_3+(s^2+1)x_2^2+2sx_2x_3-(s^2+1)x_3^2\big).
\endaligned
\right.
$$
Then the~defining equations of the~remaining surfaces $\Pi_2^+,\ldots,\Pi_{16}^+$ are listed in Figure~\ref{figure:16-surfaces}.
Now, the~intersection matrix $(\Pi_i^+\bullet \Pi_j^+)$ can be computed as follows:
$$
\left(\begin{array}{cccccccccccccccc}
-1&1&1&1&1&0&1&0&1&0&1&0&1&1&0&0\\
1&-1&1&1&0&1&0&1&1&0&1&0&0&0&1&1\\
1&1&-1&1&1&0&1&0&0&1&0&1&0&0&1&1\\
1&1&1&-1&0&1&0&1&0&1&0&1&1&1&0&0\\
1&0&1&0&-1&1&1&1&1&1&0&0&1&0&1&0\\
0&1&0&1&1&-1&1&1&0&0&1&1&1&0&1&0\\
1&0&1&0&1&1&-1&1&0&0&1&1&0&1&0&1\\
0&1&0&1&1&1&1&-1&1&1&0&0&0&1&0&1\\
1&1&0&0&1&0&0&1&-1&1&1&1&1&0&0&1\\
0&0&1&1&1&0&0&1&1&-1&1&1&0&1&1&0\\
1&1&0&0&0&1&1&0&1&1&-1&1&0&1&1&0\\
0&0&1&1&0&1&1&0&1&1&1&-1&1&0&0&1\\
1&0&0&1&1&1&0&0&1&0&0&1&-1&1&1&1\\
1&0&0&1&0&0&1&1&0&1&1&0&1&-1&1&1\\
0&1&1&0&1&1&0&0&0&1&1&0&1&1&-1&1\\
0&1&1&0&0&0&1&1&1&0&0&1&1&1&1&-1
\end{array}\right)
$$
The rank of this matrix is $7$. Therefore, we conclude that $\mathrm{Cl}^{\mathbb{H}}(X)\cong\mathbb{Z}$ by Corollary~\ref{corollary:class-group-simple-criterion}.
Note that we can also prove this using Lemma~\ref{lemma:class-group-planes}(ii).
To do this, it is enough to show that the~divisor $\Pi^+$ is a Cartier divisor, which can be done locally at any  point in $\mathrm{Sing}(X)$.
For~instance, let $P=[t:1:1:1:0]\in\mathrm{Sing}(X)$.
Among $\Pi_1^+,\ldots,\Pi_{16}^+$, only
$$
\Pi_2^+,\Pi_3^+,\Pi_7^+,\Pi_8^+,\Pi_{10}^+,\Pi_{11}^+
$$
pass through $P$. Choosing a generator of the~local class group $\mathrm{Cl}_P(X)\cong\mathbb{Z}$,
we see that the~classes of the~surfaces $\Pi_2^+$, $\Pi_3^+$, $\Pi_7^+$, $\Pi_8^+$, $\Pi_{10}^+$, $\Pi_{11}^+$ are $1$, $-1$, $1$, $-1$, $1$, $-1$, respectively.
Hence, we see that $\Pi^+$ is locally Cartier at $P$, which implies that $\Pi^+$ is globally Cartier,
because the~group $\mathbb{H}$ acts transitively on the~set $\mathrm{Sing}(X)$.
\end{example}

\begin{figure}[h!]
\caption{Defining equations of the~surfaces $\Pi_1^+,\ldots,\Pi_{16}^+$ in Example~\ref{example:class-group-planes-48-50}.}
\label{figure:16-surfaces}
\begin{center}
\renewcommand\arraystretch{1.55}
\begin{tabular}{|c||c|}
\hline
\shortstack{$\Pi_1^+$\\{ }} & \shortstack{{ }\\$x_0+x_1+x_2+\frac{2s}{s^2+1}x_3=0$\\$w=\frac{s^2-1}{(s^2+1)^2}\big((s^2+1)x_1^2+(s^2+1)x_1x_2+2sx_1x_3+(s^2+1)x_2^2+2sx_2x_3-(s^2+1)x_3^2\big)$}\\
\hline
\shortstack{$\Pi_2^+$\\{ }} & \shortstack{{ }\\{$x_0-x_1+x_2+\frac{2s}{s^2+1}x_3=0$}\\ {$w=\frac{s^2-1}{(s^2+1)^2}\big((s^2+1)x_1^2-(s^2+1)x_1x_2+2sx_1x_3+(s^2+1)x_2^2-2sx_2x_3-(s^2+1)x_3^2\big)=0$}}\\
\hline
\shortstack{$\Pi_3^+$\\{ }} & \shortstack{{ }\\{$x_0+x_1-x_2-\frac{2s}{s^2+1}x_3=0$}\\ {$w=\frac{s^2-1}{(s^2+1)^2}\big((s^2+1)x_1^2-(s^2+1)x_1x_2-2sx_1x_3+(s^2+1)x_2^2+2sx_2x_3-(s^2+1)x_3^2\big)=0$}}\\
\hline
\shortstack{$\Pi_4^+$\\{ }} & \shortstack{{ }\\{$x_0-x_1-x_2+\frac{2s}{s^2+1}x_3=0$}\\ {$w=\frac{s^2-1}{(s^2+1)^2}\big((s^2+1)x_1^2+(s^2+1)x_1x_2-2sx_1x_3+(s^2+1)x_2^2-2sx_2x_3-(s^2+1)x_3^2\big)=0$}}\\
\hline
\shortstack{$\Pi_5^+$\\{ }} & \shortstack{{ }\\{$x_0+x_1+\frac{2s}{s^2+1}x_2+x_3=0$}\\ {$w=\frac{s^2-1}{(s^2+1)^2}\big((s^2+1)x_0^2+2sx_2x_0+(s^2+1)x_3x_0-(s^2+1)x_2^2+2sx_2x_3+(s^2+1)x_3^2\big)=0$}}\\
\hline
\shortstack{$\Pi_6^+$\\{ }} & \shortstack{{ }\\{$x_0-x_1+\frac{2s}{s^2+1}x_2-x_3=0$}\\ {$w=\frac{s^2-1}{(s^2+1)^2}\big((s^2+1)x_0^2+2sx_2x_0-(s^2+1)x_3x_0-(s^2+1)x_2^2-2sx_2x_3+(s^2+1)x_3^2\big)=0$}}\\
\hline
\shortstack{$\Pi_7^+$\\{ }} & \shortstack{{ }\\{$\frac{2s}{s^2+1}x_2-x_1-x_0+x_3=0$ }\\ {$w=\frac{s^2-1}{(s^2+1)^2}\big((s^2+1)x_0^2-2sx_2x_0-(s^2+1)x_3x_0-(s^2+1)x_2^2+2sx_2x_3+(s^2+1)x_3^2\big)=0$}}\\
\hline
\shortstack{$\Pi_8^+$\\{ }} & \shortstack{{ }\\{ $x_0-x_1-\frac{2s}{s^2+1}x_2+x_3=0$}\\ {$w=\frac{s^2-1}{(s^2+1)^2}\big((s^2+1)x_0^2-2sx_2x_0+(s^2+1)x_3x_0-(s^2+1)x_2^2-2sx_2x_3+(s^2+1)x_3^2\big)=0$ }}\\
\hline
\shortstack{$\Pi_9^+$\\{ }} & \shortstack{{ }\\{ $x_0+\frac{2s}{s^2+1}x_1+x_2+x_3=0$}\\ { $w=\frac{s^2-1}{(s^2+1)^2}\big((s^2+1)x_0^2+2sx_1x_0+(s^2+1)x_3x_0-(s^2+1)x_1^2+2sx_1x_3+(s^2+1)x_3^2\big)=0$}}\\
\hline
\shortstack{$\Pi_{10}^+$\\{ }} & \shortstack{{ }\\{ $x_0-\frac{2s}{s^2+1}x_1-x_2+x_3=0$}\\ {$w=\frac{s^2-1}{(s^2+1)^2}\big((s^2+1)x_0^2-2sx_1x_0+(s^2+1)x_3x_0-(s^2+1)x_1^2-2sx_1x_3+(s^2+1)x_3^2\big)=0$ }}\\
\hline
\shortstack{$\Pi_{11}^+$\\{ }} & \shortstack{{ }\\{ $x_0-\frac{2s}{s^2+1}x_1+x_2-x_3=0$}\\ {$w=\frac{s^2-1}{(s^2+1)^2}\big((s^2+1)x_0^2-2sx_1x_0-(s^2+1)x_3x_0-(s^2+1)x_1^2+2sx_1x_3+(s^2+1)x_3^2\big)=0$ }}\\
\hline
\shortstack{$\Pi_{12}^+$\\{ }} & \shortstack{{ }\\{ $x_0+\frac{2s}{s^2+1}x_1-x_2-x_3=0$}\\ {$w=\frac{s^2-1}{(s^2+1)^2}\big((s^2+1)x_0^2+2sx_1x_0-(s^2+1)x_3x_0-(s^2+1)x_1^2-2sx_1x_3+(s^2+1)x_3^2\big)=0$ }}\\
\hline
\shortstack{$\Pi_{13}^+$\\{ }} & \shortstack{{ }\\{$\frac{2s}{s^2+1}x_0+x_1+x_2+x_3=0$ }\\ {$w=\frac{s^2-1}{(s^2+1)^2}\big((s^2+1)x_0^2-2sx_1x_0-2sx_2x_0-(s^2+1)x_1^2-(s^2+1)x_1x_2-(s^2+1)x_2^2\big)=0$ }}\\
\hline
\shortstack{$\Pi_{14}^+$\\{ }} & \shortstack{{ }\\{ $\frac{2s}{s^2+1}x_0-x_1-x_2+x_3=0$}\\ {$w=\frac{s^2-1}{(s^2+1)^2}\big((s^2+1)x_0^2+2sx_1x_0+2sx_2x_0-(s^2+1)x_1^2-(s^2+1)x_1x_2-(s^2+1)x_2^2\big)=0$ }}\\
\hline
\shortstack{$\Pi_{15}^+$\\{ }} & \shortstack{{ }\\{$\frac{2s}{s^2+1}x_0-x_1+x_2-x_3=0$ }\\ {$w=\frac{s^2-1}{(s^2+1)^2}\big((s^2+1)x_0^2+2sx_1x_0-2sx_2x_0-(s^2+1)x_1^2+(s^2+1)x_1x_2-(s^2+1)x_2^2\big)=0$ }}\\
\hline
\shortstack{$\Pi_{16}^+$\\{ }} & \shortstack{{ }\\{$\frac{2s}{s^2+1}x_0+x_1-x_2-x_3=0$ }\\ {$w=\frac{s^2-1}{(s^2+1)^2}\big((s^2+1)x_0^2-2sx_1x_0+2sx_2x_0-(s^2+1)x_1^2+(s^2+1)x_1x_2-(s^2+1)x_2^2\big)=0$ }}\\
\hline
\end{tabular}
\end{center}
\end{figure}

\begin{example}
\label{example:48-50-class-group}
Let us use assumptions and notations of Example~\ref{example:recover-curve-from-quartic}.
Then $\mathrm{Aut}(X)$ contains a unique subgroup $G$ such that $G\cong\upsilon(G)\cong\mumu_2^4\rtimes\mumu_3$,
and $\upsilon(G)$ is generated~by
$$
A_1,A_2,A_3,A_4,
\begin{pmatrix}
0 & 0 & 1 & 0\\
1 & 0 & 0 & 0\\
0 & 1 & 0 & 0\\
0 & 0 & 0 & 1
\end{pmatrix}.
$$
One can check that $G$ contains the~subgroup $\mathbb{H}=\mathbb{H}^\rho$, where $\rho$ is a trivial homomorphism.
Therefore, it follows from Example~\ref{example:class-group-planes-48-50} that $\mathrm{Cl}^G(X)\cong\mathbb{Z}$.
\end{example}

If $\mathrm{Cl}^{G}(X)\not\cong\mathbb{Z}$, we have the~following result.

\begin{lemma}
\label{lemma:congruences}
Let $G$ be any subgroup in $\mathrm{Aut}(X)$ such that $\mathrm{Cl}^G(X)\not\cong\mathbb{Z}$ and $\mathbb{H}\subseteq\upsilon(G)$.
Then $\mathrm{Cl}^{G}(X)\cong\mathbb{Z}^2$, and exists a $G$-Sarkisov link
$$
\xymatrix{
&&V\ar@{-->}[rrrr]^{\varsigma}\ar@{->}[dll]_{\phi}\ar@{->}[d]^{\varpi}&&&&V\ar@{->}[d]_{\varpi}\ar@{->}[drr]^{\phi}\\%
Z&&X\ar@{->}[rrrr]^{\sigma}&&&&X&& Z}
$$
where $\varpi$ is a $G$-equivariant small resolution,
$\varsigma$ is a birational involution that flops sixteen \mbox{$\varpi$-contracted} curves,
$\phi$ is a $\mathbb{P}^1$-bundle, and $Z$ is a smooth del Pezzo surface of degree $4$.
\end{lemma}

\begin{proof}
By Corollary~\ref{corollary:class-group-planes-H-invariant-part}, one has $\mathrm{Cl}^{G}(X)\cong\mathbb{Z}^2$, but $\mathrm{Pic}(X)\cong\mathbb{Z}$.
Using this and the~fact that $G$ acts transitively on the~set $\mathrm{Sing}(X)$,
we see that there exists a~$G$-equivariant small resolution of singularities $\varpi\colon V\to X$.
Then $-K_V\sim\varpi^*(2H)$.

Now, applying $G$-equivariant Minimal Model Program to $V$,
we obtain a $G$-equivariant morphism $\phi\colon V\to Z$ such that $\phi_{*}(\mathcal{O}_{V})=\mathcal{O}_Z$,
and one of the~following possibilities hold:
\begin{enumerate}
\item $Z$ is a possibly singular Fano threefold, and $\phi$ is birational;
\item $Z$ is a normal surface, and $\phi$ is a conic bundle;
\item $Z\cong\mathbb{P}^1$, and $\phi$ is a del Pezzo fibration.
\end{enumerate}
Note that $\mathrm{rk}\,\mathrm{Cl}^G(V/Z)=1$. In particular, the~divisor $-K_V$ is $\phi$-ample.

If $\phi$ is a del Pezzo fibration, then its general fiber is $\mathbb{P}^1\times\mathbb{P}^1$ by the~adjunction formula, because $-K_V\sim\varpi^*(2H)$.
Moreover, in this case all fibers of the~morphism $\phi$ are reduced and irreducible \cite{Fujita},
which gives $\mathrm{rk}\,\mathrm{Cl}(V)\leqslant 3$, which is impossible since  $\mathrm{Cl}(V)\cong\mathbb{Z}^7$.
Therefore, either $\phi$ is birational, or $\phi$ is a conic bundle.

Now, applying relative  Minimal Model Program to the threefold $V$ over $Z$, we obtain a~commutative diagram
$$
\xymatrix{
&V\ar@{->}[dl]_{\phi}\ar@{->}[dr]^{\alpha}&\\%
Z&&U\ar@{->}[ll]^{\beta}}
$$
where $\alpha$ is an~extremal contraction, and $\beta$ is a morphism with connected fibers.

Suppose that $\alpha$ is birational. Let $E$ be the~$\alpha$-exceptional surface.
Then $U$ is a~smooth threefold, and $\alpha$ is a usual blow up of a point in $U$ (see \cite[Lemma~2.8]{CasagrandeJahnkeRadloff} or \cite{Cutkosky,Mori}).
Then $E\cong\mathbb{P}^2$ and $E\vert_{E}\cong\mathcal{O}_{\mathbb{P}^2}(-1)$.
So, the~adjunction formula gives $\varpi^*(H)\vert_{E}\cong\mathcal{O}_{\mathbb{P}^2}(1)$.
Therefore, we conclude that $\pi\circ\varpi(E)\cong\varpi(E)\cong E\cong\mathbb{P}^2$ and $\pi\circ\varpi(E)$ is a plane in $\mathbb{P}^3$,
which implies that $\varpi(E)=\Pi_i^+$ or $\varpi(E)=\Pi_i^-$ for some $i\in\{1,\ldots,16\}$.

On the~other hand, for every $g\in G$, we have
$$
g(E)\cap E=\varnothing,
$$
because all curves in $E$ are contained in one extremal ray of the~Mori cone $\overline{\mathrm{NE}(V)}$.
Moreover, since $\mathbb{H}\subseteq\upsilon(G)$, the~$G$-orbit of the~surface $E$ consists of at least $16$ surfaces.
This implies that $\mathrm{rk}\,\mathrm{Cl}(V)\geqslant 17$,
which is impossible, since $\mathrm{Cl}(V)\cong\mathrm{Cl}(X)\cong\mathbb{Z}^7$.

Thus, we conclude that $\alpha$ is not birational, so $\phi$ is not birational either.
Hence, we see that $\phi$ is a conic bundle. Then $\alpha$ is also a conic bundle,  because otherwise $\alpha$ would be birational.
Since general fibers of both conic bundles $\alpha$ and $\beta$ are the~same,
they generate one extremal ray of the~Mori cone $\overline{\mathrm{NE}(V)}$.
This implies that $\beta$ must be an isomorphism.
Therefore, we may assume that $\alpha=\phi$ and $\beta=\mathrm{Id}_U$.

Now, it follows from \cite[Lemma~2.5]{CasagrandeJahnkeRadloff} that  $Z$ is smooth,
it is a weak del Pezzo surface, and $V\cong\mathbb{P}(\mathcal{E})$
for some rank two vector bundle $\mathcal{E}$ on the~surface $Z$.

To complete the~proof, it is enough to show that $Z$ is a del Pezzo surface of degree $4$.
First, we observe that  $\mathrm{Cl}(Z)\cong\mathbb{Z}^6$, since $\mathrm{Cl}(V)\cong\mathbb{Z}^7$.
This gives $K_Z^2=4$.

Let $C$ be a general fiber of the~$\mathbb{P}^1$-bundle $\phi$.
Then we have $2H\cdot\varpi(C)=-K_V\cdot C=2$,
which implies that $\pi\circ\phi(C)$ is a line in $\mathbb{P}^3$
that is tangent to the~surface $\mathscr{S}$ at two points.
Now, for $i\in\{1,\ldots,16\}$, let $\widetilde{\Pi}_i^\pm$ be the~proper transform on $V$ of the~surface $\Pi_i^\pm$.~Then
$$
1=\varpi^*(H)\cdot C=\big(\widetilde{\Pi}_i^++\widetilde{\Pi}_i^-\big)\cdot S=\widetilde{\Pi}_i^+\cdot C+\widetilde{\Pi}_i^-\cdot C.
$$
Thus, without loss of generality, we may assume that $\widetilde{\Pi}_i^+\cdot C=1$ and $\widetilde{\Pi}_i^-\cdot C=0$~for~every~$i$.
Then $\widetilde{\Pi}_i^-$ is mapped to a curve in $Z$,
and $\phi$ induces a birational morphism $\widetilde{\Pi}_i^+\to Z$.

Recall that $\varpi$  induces a birational morphism $\widetilde{\Pi}_i^+\to\Pi_i^+$ that is a blow up of $k\geqslant 0$ points in the~intersection $\Pi_i^+\cap\mathrm{Sing}(X)$,
where $k\leqslant 6$, since $\Pi_i^+\cap\mathrm{Sing}(X)$ consists of~six~points.
Moreover, the morphism $\varpi$ also induces a birational morphism $\widetilde{\Pi}_i^-\to\Pi_i^-$ that blows up the~remaining $6-k$ points in $\Pi_i^+\cap\mathrm{Sing}(X)=\Pi_i^-\cap\mathrm{Sing}(X)$.
We also know that the~intersection $\Pi_i^+\cap\mathrm{Sing}(X)$ is contained in an irreducible conic in $\Pi_i^+\cong\mathbb{P}^2$. Then
\begin{itemize}
\item either $k\leqslant 5$ and $\widetilde{\Pi}_i^+$ is a smooth del Pezzo surface of degree $9-k$,
\item or $k=6$ and $\widetilde{\Pi}_i^+$ is a smooth weak del Pezzo surface of degree $3$.
\end{itemize}
If $k=6$, then $\widetilde{\Pi}_i^-\cong\Pi_i^-\cong\mathbb{P}^2$, and the surface $\widetilde{\Pi}_i^-$ cannot be mapped to a curve by $\phi$.
Thus,~we conclude that $k\leqslant 5$,
and $\widetilde{\Pi}_i^+$ is a smooth del Pezzo surface of degree $9-k\geqslant 4$.
Since $\phi$ induces a birational morphism $\widetilde{\Pi}_i^+\to Z$ and $K_Z^2=4$,
we get $k=5$ and $\widetilde{\Pi}_i^+\cong Z$, which implies that $Z$ is a del Pezzo surface of degree $4$,
and $\phi(\widetilde{\Pi}_i^-)$ is a $(-1)$-curve.
\end{proof}

Before we proceed, let us make few remarks:
\begin{enumerate}
\item
In the proof of Lemma~\ref{lemma:congruences}, we show that general fibers of the~$\mathbb{P}^1$-bundle $\phi$
are mapped by the morphism $\pi\circ\varpi$ to the~lines in $\mathbb{P}^3$ tangent to $\mathscr{S}$ at two points.
These lines form a~classical congruence in $\mathbb{P}^3$ of degree $(2,2)$,
which is related to the~quartic surface $\mathscr{S}$, see \cite{Kummer1865,Kummer1866,Caporali,Jessop,GeemenPreviato,Dolgachev2020}, \cite[Example~5.5]{ArrondoBertoliniTurrini},~\cite[Example~2.3]{PonceSturmfelsTrager}.
This implies that the weak Fano threefold $V$ in Lemma~\ref{lemma:congruences} is isomorphic to $\mathbb{P}(\mathcal{E})$,
where $\mathcal{E}$ is a rank two vector bundle on the quartic del Pezzo surface $Z\subset\mathrm{Gr}(2,4)$,
which is obtained by restricting to $Z$ the universal quotient bundle from  $\mathrm{Gr}(2,4)$.

\item It was pointed to us by Alexander Kuznetsov that the~diagram in Lemma~\ref{lemma:congruences} can be constructed as follows.
We know from \cite{SzurekWisniewski} the existence of the following diagram:
$$
\xymatrix{
&&\mathbb{P}\big(\mathscr{N}\big)\cong\mathbb{P}\big(\mathscr{E}\big)\ar@{->}[dll]_{\pi_{\mathbb{P}^3}}\ar@{->}[drr]^{\pi_{\mathcal{Q}}}&&\\%
\mathbb{P}^3&&&&\mathcal{Q}}
$$
where $\mathscr{N}$ is the rank two null-correlation vector bundle on $\mathbb{P}^3$,
$\mathcal{Q}$ is a smooth quadric threefold,
$\mathscr{E}$ is a rank two vector bundle on the quadric $\mathcal{Q}\subset\mathrm{Gr}(2,4)$ obtained
by restricting to $\mathcal{Q}$ the universal quotient bundle from the Grassmannian,
and $\pi_{\mathbb{P}^3}$ and $\pi_{\mathcal{Q}}$ are projections.
Thus, since $Z$ is cut out on $\mathcal{Q}$ by another quadric,
we can identify $V$ with a preimage of the del Pezzo surface $Z$ in the~fourfold~$\mathbb{P}(\mathscr{E})$.
Then $\pi_{\mathbb{P}^3}$ induces a morphism $V\to\mathbb{P}^3$, which is the composition $\pi\circ\varpi$.

\item We can also describe the~$\mathbb{P}^1$-bundle $\phi$ in Lemma~\ref{lemma:congruences} using the~diagram~\eqref{equation:Prokhorov}.
Namely, a general fiber of this~$\mathbb{P}^1$-bundle is mapped by $\eta\circ\varphi^{-1}\circ\varpi$ to
\begin{itemize}
\item either a line that passes through one of the points $P_1,\ldots,P_6$,
\item or a twisted cubic curve that passes through five points among $P_1,\ldots,P_6$.
\end{itemize}

\item If $\widetilde{\mathscr{S}}$ is the preimage of the quartic surface $\mathscr{S}$ on the threefold $V$ in Lemma~\ref{lemma:congruences},
then $\widetilde{\mathscr{S}}$ is smooth, the~morphism $\pi\circ\varpi$ induces a minimal resolution
$\widetilde{\mathscr{S}}\to\mathscr{S}$, and the $\mathbb{P}^1$-bundle $\phi$
induces a $G$-equivariant double cover $\widetilde{\mathscr{S}}\to Z$, cf. \cite{Skorobogatov}.

\item Using Lemma~\ref{lemma:congruences} one can relate the subgroup $G\subset\mathrm{Aut}(X)$ such that $\mathrm{Cl}^{G}(X)\cong\mathbb{Z}^2$
with a corresponding subgroup in the group $\mathrm{Aut}(Z)$, see \cite{Hosoh,DolgachevIskovskikh}.

\end{enumerate}

Note that $\mathrm{Cl}^{G}(X)\cong\mathbb{Z}^2$ is indeed possible. Let us give two (related) examples.

\begin{example}
\label{example:Heisenberg-class-group-Z-Z}
Let us use all assumptions and notations of Example~\ref{example:class-group-planes-48-50},
and let $G=\mathbb{H}^\rho$, where the~homomorphism $\rho$ is defined by $\rho(A_1)=-1$, $\rho(A_2)=1$, $\rho(A_3)=-1$, $\rho(A_4)=1$.
Then, arguing as in Example~\ref{example:class-group-planes-48-50}, we compute $\mathrm{Cl}^{G}(X)\cong\mathbb{Z}^2$.
\end{example}

\begin{example}
\label{example:80-49}
Let us use all assumptions and notations of Example~\ref{example:Z5}.
Then
$$
\mathrm{Aut}(X)\cong \mumu_2\times\mathrm{Aut}\big(\mathbb{P}^3,\mathscr{S}\big)\cong\mumu_2\times\big(\mumu_2^4\rtimes\mumu_5\big),
$$
and the~group $\mathrm{Aut}(X)$ contains a unique subgroup isomorphic to $\mathrm{Aut}(\mathbb{P}^3,\mathscr{S})\cong\mumu_2^4\rtimes\mumu_5$.
Suppose that $G$ is this subgroup.
It follows from Remark~\ref{remark:Weyl-group}
that $\mathrm{Cl}(X)\otimes\mathbb{Q}$ is~a~faithful seven-dimensional $G$-representation.
Using this, it is easy to see that $\mathrm{Cl}(X)\otimes\mathbb{Q}$~splits
as a~sum of an~irreducible five-dimensional representation and two trivial one-dimensional representations.
Hence, we conclude that $\mathrm{Cl}^{G}(X)\cong\mathbb{Z}^2$.
\end{example}

Before proving Theorem~\ref{theorem:main}, let us prove its two baby cases, which follow from \cite{CheltsovShramov2017,CheltsovSarikyan}.

\begin{proposition}
\label{proposition:80-49}
Suppose $G=\mathrm{Aut}(X)$, and $\mathscr{S}$ is the~quartic surface from Example~\ref{example:Z5}. Then $\mathrm{Cl}^G(X)\cong\mathbb{Z}$ and $X$ is $G$-birationally super-rigid.
\end{proposition}

\begin{proof}
Since $\sigma\in G$, we get $\mathrm{Cl}^G(X)\cong\mathbb{Z}$. Let us show that $X$ is $G$-birationally super-rigid.

Note that the~$\upsilon(G)$-equivariant birational geometry of the~projective space $\mathbb{P}^3$ has been studied~in~\cite{CheltsovShramov2017}.
In~particular, we know from \cite[Corollary~4.7]{CheltsovShramov2017} and \cite[Theorem 4.16]{CheltsovShramov2017}~that
\begin{itemize}
\item $\mathbb{P}^3$ does not contain $\upsilon(G)$-orbits of length less that $16$,
\item $\mathbb{P}^3$ does not contain $\upsilon(G)$-invariant curves of degree less than $8$.
\end{itemize}

Let $\mathcal{M}$ be a $G$-invariant linear system on $X$ such that $\mathcal{M}$ has no fixed components.
Choose a~positive integer $n$ such that $\mathcal{M}\subset |nH|$.
Then, by \cite[Corollary~3.3.3]{CheltsovShramov}, to prove that the~threefold $X$ is $G$-birationally super-rigid it is enough to show that $(X,\frac{2}{n}\mathcal{M})$ has canonical singularities.
Suppose that the~singularities of this log pair are not canonical.

Let $Z$ be a center of non-canonical singularities of the~pair $(X,\frac{2}{n}\mathcal{M})$ that has the~largest dimension.
Since the~linear system $\mathcal{M}$ does not have fixed components,
we conclude that either $Z$ is an irreducible curve, or $Z$ is a point.
In both cases, we have
$$
\mathrm{mult}_Z\big(\mathcal{M}\big)>\frac{n}{2}
$$
by \cite[Theorem~4.5]{KoMo98}.

Let $M_1$ and $M_2$ be general surfaces in $\mathcal{M}$. If $Z$ is a curve, then
$$
M_1\cdot M_2=(M_1\cdot M_2)_Z\mathscr{Z}+\Delta
$$
where $\mathscr{Z}$ is the~$G$-irreducible curve in $X$ whose irreducible component is the~curve $Z$,
and~$\Delta$ is an effective one-cycle whose support does not contain $\mathscr{Z}$, which gives
\begin{multline*}
2n^2=n^2H^2=H\cdot M_1\cdot M_2=\big(M_1\cdot M_2\big)_Z\mathscr{Z}+\Delta=\\
=\big(M_1\cdot M_2\big)_Z\big(H\cdot \mathscr{Z}\big)+H\cdot\Delta\geqslant\big(M_1\cdot M_2\big)_Z(H\cdot \mathscr{Z})\geqslant\\
\geqslant\mathrm{mult}_Z^2\big(\mathcal{M}\big)\big(H\cdot \mathscr{Z}\big)>\frac{n^2}{4}\big(H\cdot \mathscr{Z}\big)\geqslant \frac{n^2}{4}\mathrm{deg}\big(\pi(\mathscr{Z})\big),
\end{multline*}
so $\pi(\mathscr{Z})$ is a $\upsilon(G)$-invariant curve of degree $\leqslant 7$,
which contradicts \cite[Theorem 4.16]{CheltsovShramov2017}.

We see that $Z$ is a point, and $(X,\frac{2}{n}\mathcal{M})$ is canonical away from finitely many points.

We claim that $Z\not\in\mathrm{Sing}(X)$.
Indeed, suppose  $Z$ is a singular point of the~threefold~$X$.
Let $h\colon\overline{X}\to X$ be the~blow up of the~locus $\mathrm{Sing}(X)$,
let $E_1,\ldots,E_{16}$ be the~$h$-exceptional surfaces,
let $\overline{M}_1$ and $\overline{M}_1$ be the~proper transforms on $\overline{X}$ of the~surfaces $M_1$ and $M_2$, respectively.
Write $E=E_1+\cdots+E_{16}$.
Since $\mathrm{Sing}(X)$ is a $G$-orbit, we have
$$
\overline{M}_1\sim\overline{M}_2\sim h^{*}(H)-\epsilon E
$$
for some integer $\epsilon\geqslant 0$. Using \cite[Theorem~3.10]{Co00} or \cite[Theorem~1.7.20]{CheltsovUMN}, we~get $\epsilon>\frac{n}{2}$.
On the~other hand, the~linear system $|h^{*}(3H)-E|$ is not empty and does not have
base curves away from the~locus $E_1\cup E_2\cup\cdots\cup E_{16}$, because $\mathrm{Sing}(\mathscr{S})$ is cut out by cubic surfaces in $\mathbb{P}^3$.
In particular, the~divisor $h^{*}(3H)-E$ is nef, so
$$
0\leqslant \big(h^{*}(3H)-E\big)\cdot \overline{M}_1\cdot\overline{M}_2=\big(h^{*}(3H)-E\big)\cdot\big(h^{*}(3nH)-\epsilon  E\big)^2=6n^2-32\epsilon^2,
$$
which is impossible, since $\epsilon >\frac{n}{2}$. So, we see that $Z$ is a smooth point of the~threefold~$X$.

Then the~pair $(X,\frac{3}{n}\mathcal{M})$ is not log canonical at $Z$.
Let $\mu$ be the~largest rational number such that the~log pair $(X,\mu\mathcal{M})$ is log canonical.
Then $\mu<\frac{3}{n}$ and
$$
\mathrm{Orb}_G(Z)\subseteq\mathrm{Nklt}\big(X,\mu\mathcal{M}\big).
$$
Observe that $\mathrm{Nklt}(X,\mu\mathcal{M})$ is at most one-dimensional, since $\mathcal{M}$ has no fixed components.
Moreover, this locus is $G$-invariant, because $\mathcal{M}$ is $G$-invariant.

We claim that $\mathrm{Nklt}(X,\mu\mathcal{M})$ does not contain curves.
Indeed, suppose this is not true. Then $\mathrm{Nklt}(X,\mu\mathcal{M})$ contains a $G$-irreducible curve~$C$.
We write $M_1\cdot M_2=mC+\Omega$, where $m$ is a non-negative integer, and $\Omega$ is an effective one-cycle whose support does not contain the~curve $C$.
Then it follows from \cite[Theorem~3.1]{Co00} that
$$
m\geqslant \frac{4}{\mu^2}>\frac{4n^2}{9}.
$$
Therefore, we have
$$
2n^2=n^2H^3=H\cdot M_1\cdot M_2=m\big(H\cdot C\big)+H\cdot\Omega\geqslant m\big(H\cdot C\big)>\frac{4n^2}{9}\big(H\cdot C\big),
$$
which implies that $H\cdot C\leqslant 4$. Then $\pi(C)$ is a $\upsilon(G)$-invariant curve in $\mathbb{P}^3$ of degree $\leqslant 4$,
which contradicts \cite[Theorem 4.16]{CheltsovShramov2017}.
Thus, the~locus $\mathrm{Nklt}(X,\mu\mathcal{M})$ contains no curves.

Let $\mathcal{I}$ be the~multiplier ideal sheaf of the~pair $(X,\mu\mathcal{M})$,
and let $\mathcal{L}$ be the~corresponding subscheme in $X$.
Then $\mathcal{L}$ is a zero-dimensional (reduced) subscheme such that
$$
\mathrm{Orb}_G\big(Z\big)\subseteq\mathrm{Supp}\big(\mathcal{L}\big)=\mathrm{Nklt}\big(X,\mu\mathcal{M}\big).
$$
Applying Nadel's vanishing \cite[Theorem~9.4.8]{Lazarsfeld}, we get $h^1(X,\mathcal{I}\otimes\mathcal{O}_{X}(H))=0$.
This gives
$$
4=h^0\big(X,\mathcal{O}_{X}(H)\big)\geqslant h^0\big(\mathcal{O}_{\mathcal{L}}\otimes\mathcal{O}_{X}(H)\big)=h^0\big(\mathcal{O}_{\mathcal{L}}\big)\geqslant|\mathrm{Orb}_G(Z)|.
$$
In particular, we conclude that the~length of the~$\upsilon(G)$-orbit of the~point $\pi(Z)$ is at most~$4$, which is impossible by \cite[Corollary~4.7]{CheltsovShramov2017}.
\end{proof}

\begin{proposition}
\label{proposition:48-50}
Suppose  that $\mathscr{S}$ is the~surface from Example~\ref{example:recover-curve-from-quartic},
and $G$ is the~subgroup described in Example~\ref{example:48-50-class-group}.
Then $\mathrm{Cl}^G(X)\cong\mathbb{Z}$ and $X$ is $G$-birationally super-rigid.
\end{proposition}

\begin{proof}
Recall from Example~\ref{example:48-50-class-group} that $G\cong\mumu_2^4\rtimes\mumu_3$ and $\mathrm{Cl}^G(X)\cong\mathbb{Z}$.

The $\upsilon(G)$-equivariant geometry of the~projective space $\mathbb{P}^3$ has been studied~in~\cite{CheltsovSarikyan}.
In~particular, we know from \cite{CheltsovSarikyan} that $\mathbb{P}^3$ does not contain $\upsilon(G)$-orbits of length $1$, $2$ or $3$,
and the~only $\upsilon(G)$-orbits in $\mathbb{P}^3$ of length~$4$ are
\begin{align*}
\Sigma_4&=\big\{[1:0:0:0],[0:1:0:0],[0:0:1:0],[0:0:0:1]\big\},\\
\Sigma_4^\prime&=\big\{[1:1:1:-1], [1:1:-1:1], [1:-1:1:1], [-1:1:1:1]\big\},\\
\Sigma_4^{\prime\prime}&=\big\{[1:1:1:1], [1:1:-1:-1], [1:-1:-1:1], [-1:-1:1:1]\big\}
\end{align*}
We also know from \cite{CheltsovSarikyan} the~classification of $\upsilon(G)$-invariant curves in $\mathbb{P}^3$ of degree at most~$7$.
Namely, let $\mathcal{L}_4$, $\mathcal{L}_4^\prime$, $\mathcal{L}_4^{\prime\prime}$,
$\mathcal{L}_4^{\prime\prime\prime}$, $\mathcal{L}_6$, $\mathcal{L}_6^\prime$, $\mathcal{L}_6^{\prime\prime}$,
$\mathcal{L}_6^{\prime\prime\prime}$, $\mathcal{L}_6^{\prime\prime\prime\prime}$ be ~$\upsilon(G)$-irreducible curves in $\mathbb{P}^3$
whose irreducible components are the~lines
\begin{align*}
\big\{2x_0+(1+\sqrt{3}i)x_2-(1-\sqrt{3}i)x_3=2x_1+(1-\sqrt{3}i)x_2+(1+\sqrt{3}i)x_3=0\big\},\\
\big\{2x_0+(1-\sqrt{3}i)x_2-(1+\sqrt{3}i)x_3=2x_1+(1+\sqrt{3}i)x_2+(1-\sqrt{3}i)x_3=0\big\},\\
\big\{2x_0-(1-\sqrt{3}i)x_2+(1+\sqrt{3}i)x_3=2x_1+(1+\sqrt{3}i)x_2+(1-\sqrt{3}i)x_3=0\big\},\\
\big\{2x_0-(1+\sqrt{3}i)x_2+(1-\sqrt{3}i)x_3=2x_1+(1-\sqrt{3}i)x_2+(1+\sqrt{3}i)x_3=0\big\},\\
\big\{x_0=x_1=0\big\},\\
\big\{x_0+x_1=x_2-x_3=0\big\},\\
\big\{x_0+x_1=x_2+x_3=0\big\},\\
\big\{x_0+ix_2=x_1+ix_3=0\big\},\\
\big\{x_0+ix_3=x_1+ix_2=0\big\},
\end{align*}
respectively. Then the~curves $\mathcal{L}_4$, $\mathcal{L}_4^\prime$, $\mathcal{L}_4^{\prime\prime}$, $\mathcal{L}_4^{\prime\prime\prime}$,
$\mathcal{L}_6$, $\mathcal{L}_6^\prime$, $\mathcal{L}_6^{\prime\prime}$,
$\mathcal{L}_6^{\prime\prime\prime}$, $\mathcal{L}_6^{\prime\prime\prime\prime}$ are unions of $4$, $4$, $4$, $4$, $6$, $6$, $6$, $6$, $6$ lines, respectively.
Moreover, it follows from \cite{CheltsovSarikyan} that
\begin{center}
$\mathcal{L}_4$, $\mathcal{L}_4^\prime$, $\mathcal{L}_4^{\prime\prime}$, $\mathcal{L}_4^{\prime\prime\prime}$,
$\mathcal{L}_6$, $\mathcal{L}_6^\prime$, $\mathcal{L}_6^{\prime\prime}$,
$\mathcal{L}_6^{\prime\prime\prime}$, $\mathcal{L}_6^{\prime\prime\prime\prime}$
\end{center}
are the~only $\upsilon(G)$-invariant curves in $\mathbb{P}^3$ of degree at most $7$.

Now, using the~defining equation of the~surface $\mathscr{S}$, one can check that any irreducible component of any curve among
$\mathcal{L}_4$, $\mathcal{L}_4^\prime$, $\mathcal{L}_4^{\prime\prime}$, $\mathcal{L}_4^{\prime\prime\prime}$,
$\mathcal{L}_6$, $\mathcal{L}_6^\prime$, $\mathcal{L}_6^{\prime\prime}$,
$\mathcal{L}_6^{\prime\prime\prime}$, $\mathcal{L}_6^{\prime\prime\prime\prime}$ intersects the~quartic surface  $\mathscr{S}$
transversally by $4$ distinct points, so that its preimage in $X$ via the~double cover $\pi$ is a smooth elliptic curve.
Thus, if $C$ is a $G$-invariant curve in $X$, then $H\cdot C\geqslant 8$.

Suppose that $X$ is not $G$-birationally super-rigid.
It follows from \cite[Corollary~3.3.3]{CheltsovShramov} that there are a~positive integer $n$ and a $G$-invariant linear subsystem $\mathcal{M}\subset |nH|$
such that~$\mathcal{M}$ does not have fixed components, but the~log pair  $(X,\frac{2}{n}\mathcal{M})$ is not canonical.

Arguing as in the~proof of Proposition~\ref{proposition:80-49},
we see that the~log pair $(X,\frac{2}{n}\mathcal{M})$ is canonical away from finitely many points.
Let $P$ be a point in $X$ that is a center of non-canonical singularities of the~log pair $(X,\frac{2}{n}\mathcal{M})$.
Now, arguing as in the~proof of Proposition~\ref{proposition:80-49}~again,
we see that $P$ is a smooth point of the~threefold $X$.

Then the~log pair $(X,\frac{3}{n}\mathcal{M})$ is not log canonical at $P$.
Let $\mu$ be the~largest rational number such that  $(X,\mu\mathcal{M})$ is log canonical.
Then $\mu<\frac{3}{n}$ and
$$
\mathrm{Orb}_G(P)\subseteq\mathrm{Nklt}\big(X,\mu\mathcal{M}\big).
$$
Observe that $\mathrm{Nklt}(X,\mu\mathcal{M})$ is at most one-dimensional, since $\mathcal{M}$ has no fixed components.
Moreover, this locus is $G$-invariant, because $\mathcal{M}$ is $G$-invariant.
Furthermore, arguing as in the~proof of Proposition~\ref{proposition:80-49},
we see that
$$
\mathrm{dim}\Big(\mathrm{Nklt}\big(X,\mu\mathcal{M}\big)\Big)=0.
$$

Let $\mathcal{I}$ be the~multiplier ideal sheaf of the~pair $(X,\mu\mathcal{M})$,
and let $\mathcal{L}$ be the~corresponding subscheme in $X$.
Then $\mathcal{L}$ is a zero-dimensional (reduced) subscheme such that
$$
\mathrm{Orb}_G\big(P\big)\subseteq\mathrm{Supp}\big(\mathcal{L}\big)=\mathrm{Nklt}\big(X,\mu\mathcal{M}\big).
$$
On the~other hand, applying Nadel's vanishing theorem \cite[Theorem~9.4.8]{Lazarsfeld}, we get
$$
h^1\big(X,\mathcal{I}\otimes\mathcal{O}_{X}(H)\big)=0.
$$
This gives
$$
4=h^0\big(X,\mathcal{O}_{X}(H)\big)\geqslant h^0\big(\mathcal{O}_{\mathcal{L}}\otimes\mathcal{O}_{X}(H)\big)=h^0\big(\mathcal{O}_{\mathcal{L}}\big)\geqslant|\mathrm{Orb}_G(P)|.
$$
Thus, we conclude that $|\mathrm{Orb}_G(P)|=4$ and $\pi(P)\in\Sigma_4\cup\Sigma_4^\prime\cup\Sigma_4^{\prime\prime}$.

Let $M_1$ and $M_2$ be two general surfaces in $\mathcal{M}$. Using \cite{Pukhlikov} or \cite[Corollary 3.4]{Co00}, we get
\begin{equation}
\label{equation:Corti}
\big(M_1\cdot M_2\big)_P>n^2.
\end{equation}
Let $\mathcal{S}$ be a linear subsystem in $|3H|$ that consists of all surfaces that are singular at every point of the~$G$-orbit $\mathrm{Orb}_G(P)$.
Then its base locus does not contain curves,
which implies that there is a surface $S\in\mathcal{S}$ that does not contain components of the~cycle~$M_1\cdot M_2$.
Thus, using \eqref{equation:Corti} and $\mathrm{mult}_P(S)\geqslant 2$, we get
$$
6n^2=S\cdot M_1\cdot M_2\geqslant\sum_{O\in \mathrm{Orb}_G(P)}2\big(M_1\cdot M_2\big)_O=2|\mathrm{Orb}_G(P)|\big(M_1\cdot M_2\big)_P=8\big(M_1\cdot M_2\big)_P>8n^2,
$$
which is absurd. This completes the~proof of Proposition~\ref{proposition:48-50}.
\end{proof}

In the~remaining part of the~paper, we prove Theorem~\ref{theorem:main},
and consider one application.
Let us recall from \cite{Hudson,Nieto,Gonzalez-Dorrego,Eklund,Antonsen} basic facts about the~$\mathbb{H}$-equivariant geometry of $\mathbb{P}^3$.~Set
\begin{align*}
\mathcal{Q}_1&=\big\{x_0^2+x_1^2+x_2^2+x_3^2=0\big\},\\
\mathcal{Q}_2&=\big\{x_0^2+x_1^2=x_2^2+x_3^2\big\},\\
\mathcal{Q}_3&=\big\{x_0^2-x_1^2=x_2^2-x_3^2\big\},\\
\mathcal{Q}_4&=\big\{x_0^2-x_1^2=x_3^2-x_2^2\big\},\\
\mathcal{Q}_5&=\big\{x_0x_2+x_1x_3=0\big\},\\
\mathcal{Q}_6&=\big\{x_0x_3+x_1x_2=0\big\}, \\
\mathcal{Q}_7&=\big\{x_0x_1+x_2x_3=0\big\},\\
\mathcal{Q}_8&=\big\{x_0x_2=x_1x_3\big\}, \\
\mathcal{Q}_9&=\big\{x_0x_3=x_1x_2\big\}, \\
\mathcal{Q}_{10}&=\big\{x_0x_1=x_2x_3\big\}.
\end{align*}
Then $\mathcal{Q}_1$, $\mathcal{Q}_2$, $\mathcal{Q}_3$, $\mathcal{Q}_4$, $\mathcal{Q}_5$, $\mathcal{Q}_6$, $\mathcal{Q}_7$,
$\mathcal{Q}_8$, $\mathcal{Q}_9$, $\mathcal{Q}_{10}$ are all $\mathbb{H}$-invariant quadric surfaces in~$\mathbb{P}^3$.
These quadrics are smooth, and $\mathbb{H}\cong\mumu_2^4$ acts \emph{naturally} on each quadric $\mathcal{Q}_i\cong\mathbb{P}^1\times\mathbb{P}^1$.

For a non-trivial element $g\in\mathbb{H}$,
the locus of its fixed points in $\mathbb{P}^3$ consists of two skew lines, which we will denote by $L_g$ and $L_g^\prime$.
For two non-trivial elements $g\ne h$ in $\mathbb{H}$, one~has
$$
\big\{L_g,L_g^\prime\big\}\cap\big\{L_h,L_h^\prime\big\}=\varnothing.
$$
In total, this gives $30$ lines $\ell_{1},\ldots,\ell_{30}$, whose equations are listed in the~following table:
\begin{center}
\renewcommand\arraystretch{1.5}
\begin{tabular}{|c||c|}
\hline
$\quad$ $\ell_1=\big\{x_0=x_1=0\big\}$ $\quad$  &$\quad$ $\ell_2=\big\{x_2=x_3=0\big\}$ $\quad$  \\
\hline
$\quad$ $\ell_3=\big\{x_0=x_2=0\big\}$ $\quad$ & $\quad$ $\ell_4=\big\{x_1=x_3=0\big\}$ $\quad$ \\
\hline
$\quad$ $\ell_5=\big\{x_0=x_3=0\big\}$ $\quad$ & $\quad$ $\ell_6=\big\{x_1=x_2=0\big\}$ $\quad$ \\
\hline
$\quad$ $\ell_7=\big\{x_0+x_1=x_2+x_3=0\big\}$ $\quad$ & $\quad$ $\ell_8=\big\{x_0-x_1=x_2-x_3=0\big\}$ $\quad$ \\
\hline
$\quad$ $\ell_9=\big\{x_0+x_2=x_1+x_3=0\big\}$ $\quad$ & $\quad$ $\ell_{10}=\big\{x_0-x_2=x_1-x_3=0\big\}$ $\quad$ \\
\hline
$\quad$ $\ell_{11}=\big\{x_0+x_3=x_1+x_2=0\big\}$ $\quad$ & $\quad$ $\ell_{12}=\big\{x_0-x_3=x_1-x_2=0\big\}$ $\quad$ \\
\hline
$\quad$ $\ell_{13}=\big\{x_0+x_1=x_2-x_3=0\big\}$ $\quad$ & $\quad$ $\ell_{14}=\big\{x_0-x_1=x_2+x_3=0\big\}$ $\quad$ \\
\hline
$\quad$ $\ell_{15}=\big\{x_0+x_2=x_1-x_3=0\big\}$ $\quad$ & $\quad$ $\ell_{16}=\big\{x_0-x_2=x_1+x_3=0\big\}$ $\quad$ \\
\hline
$\quad$ $\ell_{17}=\big\{x_0+x_3=x_1-x_2=0\big\}$ $\quad$ &$\quad$ $\ell_{18}=\big\{x_0-x_3=x_1+x_2=0\big\}$ $\quad$  \\
\hline
$\quad$ $\ell_{19}=\big\{x_0+ix_1=x_2+ix_3=0\big\}$ $\quad$ &$\quad$  $\ell_{20}=\big\{x_0-ix_1=x_2-ix_3=0\big\}$ $\quad$ \\
\hline
$\quad$ $\ell_{21}=\big\{x_0+ix_2=x_1+ix_3=0\big\}$ $\quad$ &$\quad$  $\ell_{22}=\big\{x_0-ix_2=x_1-ix_3=0\big\}$ $\quad$ \\
\hline
$\quad$ $\ell_{23}=\big\{x_0+ix_3=x_1+ix_2=0\big\}$ $\quad$ &$\quad$  $\ell_{24}=\big\{x_0-ix_3=x_1-ix_2=0\big\}$ $\quad$ \\
\hline
$\quad$ $\ell_{25}=\big\{x_0-ix_1=x_2+ix_3=0\big\}$ $\quad$ &$\quad$  $\ell_{26}=\big\{x_0+ix_1=x_2-ix_3=0\big\}$ $\quad$ \\
\hline
$\quad$ $\ell_{27}=\big\{x_0+ix_2=x_1-ix_3=0\big\}$ $\quad$ &$\quad$  $\ell_{28}=\big\{x_0-ix_2=x_1+ix_3=0\big\}$ $\quad$ \\
\hline
$\quad$ $\ell_{29}=\big\{x_0+ix_3=x_1-ix_2=0\big\}$ $\quad$ &$\quad$  $\ell_{30}=\big\{x_0-ix_3=x_1+ix_2=0\big\}$ $\quad$ \\
\hline
\end{tabular}
\end{center}

Note that $\ell_{1},\ldots,\ell_{30}$ are irreducible components of the~curves $\mathcal{L}_6$, $\mathcal{L}_6^\prime$, $\mathcal{L}_6^{\prime\prime}$,
$\mathcal{L}_6^{\prime\prime\prime}$, $\mathcal{L}_6^{\prime\prime\prime\prime}$ which have been introduced in the~proof of Proposition~\ref{proposition:48-50}.
One can check that
\begin{itemize}
\item for every $k\in\{1,\ldots,15\}$, the~curve $\ell_{2k-1}+\ell_{2k}$ is $\mathbb{H}$-irreducible,
\item each line among $\ell_{1},\ldots,\ell_{30}$ is contained in $4$ quadrics among $\mathcal{Q}_1,\ldots,\mathcal{Q}_{10}$,
\item each quadric among $\mathcal{Q}_1,\ldots,\mathcal{Q}_{10}$ contains $12$ lines among $\ell_{1},\ldots,\ell_{30}$,
\item every two quadrics among $\mathcal{Q}_1,\ldots,\mathcal{Q}_{10}$ intersect by $4$ lines among $\ell_{1},\ldots,\ell_{30}$.
\end{itemize}
The incidence relation between $\ell_{1},\ldots,\ell_{30}$ and  $\mathcal{Q}_1,\ldots,\mathcal{Q}_{10}$ is presented in Figure~\ref{figure:10-quadrics}.

\begin{figure}[h!]
\caption{Ten $\mathbb{H}$-invariant quadrics in $\mathbb{P}^3$ and thirty lines in them.}
\label{figure:10-quadrics}
\begin{center}
\renewcommand\arraystretch{1.4}
\begin{tabular}{|c||c|c|c|c|c|c|c|c|c|c|}
\hline
$\quad$ $\enskip$  & $\enskip$ $\mathcal{Q}_1$ $\enskip$  & $\enskip$ $\mathcal{Q}_2$ $\enskip$ & $\enskip$ $\mathcal{Q}_3$ $\enskip$  & $\enskip$ $\mathcal{Q}_4$ $\enskip$ & $\enskip$ $\mathcal{Q}_5$ $\enskip$ & $\enskip$ $\mathcal{Q}_6$ $\enskip$ & $\enskip$ $\mathcal{Q}_7$ $\enskip$ & $\enskip$  $\mathcal{Q}_{8}$ $\enskip$ & $\enskip$  $\mathcal{Q}_9$ $\enskip$ & $\enskip$ $\mathcal{Q}_{10}$ $\enskip$\\
\hline\hline
$\ell_1$                         & $-$  & $-$  & $-$  & $-$  & $+$  & $+$  & $-$  & $+$  & $+$  & $-$ \\
\hline
$\ell_2$                         & $-$  & $-$  & $-$  & $-$  & $+$  & $+$  & $-$  & $+$  & $+$  & $-$ \\
\hline
$\ell_3$                         & $-$  & $-$  & $-$  & $-$  & $-$  & $+$  & $+$  & $-$  & $+$  & $+$ \\
\hline
$\ell_4$                         & $-$  & $-$  & $-$  & $-$  & $-$  & $+$  & $+$  & $-$  & $+$  & $+$\\
\hline
$\ell_5$                         & $-$  & $-$  & $-$  & $-$  & $+$  & $-$  & $+$  & $+$  & $-$  & $+$ \\
\hline
$\ell_6$                         & $-$  & $-$  & $-$  & $-$  & $+$  & $-$  & $+$  & $+$  & $-$  & $+$\\
\hline
$\ell_7$                         & $-$  & $-$  & $+$  & $+$  & $-$  & $-$  & $-$  & $+$  & $+$  & $-$ \\
\hline
$\ell_8$                         & $-$  & $-$  & $+$  & $+$  & $-$  & $-$  & $-$  & $+$  & $+$  & $-$ \\
\hline
$\ell_9$                         & $-$  & $+$  & $+$  & $-$  & $-$  & $-$  & $-$  & $-$  & $+$  & $+$ \\
\hline
$\ell_{10}$                      & $-$  & $+$  & $+$  & $-$  & $-$  & $-$  & $-$  & $-$  & $+$  & $+$ \\
\hline
$\ell_{11}$                      & $-$  & $+$  & $-$  & $+$  & $-$  & $-$  & $-$  & $+$  & $-$  &  $+$\\
\hline
$\ell_{12}$                      & $-$  & $+$  & $-$  & $+$  & $-$  & $-$  & $-$  & $+$  & $-$  & $+$ \\
\hline
$\ell_{13}$                      & $-$  & $-$  & $+$  & $+$  & $+$  & $+$  & $-$  & $-$  & $-$  & $-$ \\
\hline
$\ell_{14}$                      & $-$  & $-$  & $+$  & $+$  & $+$  & $+$  & $-$  & $-$  & $-$  & $-$ \\
\hline
$\ell_{15}$                      & $-$  & $+$  & $+$  & $-$  & $-$  & $+$  & $+$  & $-$  & $-$  &  $-$\\
\hline
$\ell_{16}$                      & $-$  & $+$  & $+$  & $-$  & $-$  & $+$  & $+$  & $-$  & $-$  & $-$ \\
\hline
$\ell_{17}$                      & $-$  & $+$  & $-$  & $+$  & $+$  & $-$  & $+$  & $-$  & $-$  & $-$\\
\hline
$\ell_{18}$                      & $-$  & $+$  & $-$  & $+$  & $+$  & $-$  & $+$  & $-$  & $-$  &  $-$\\
\hline
$\ell_{19}$                      & $+$  & $+$  & $-$  & $-$  & $+$  & $-$  & $-$  & $-$  & $+$  & $-$ \\
\hline
$\ell_{20}$                      & $+$  & $+$  & $-$  & $-$  & $+$  & $-$  & $-$  & $-$  & $+$  & $-$ \\
\hline
$\ell_{21}$                      & $+$  & $-$  & $-$  & $+$  & $-$  & $-$  & $+$  & $-$  & $+$  & $-$ \\
\hline
$\ell_{22}$                      & $+$  & $-$  & $-$  & $+$  & $-$  & $-$  & $+$  & $-$  & $+$  & $-$ \\
\hline
$\ell_{23}$                      & $+$  & $-$  & $+$  & $-$  & $+$  & $-$  & $+$  & $+$  & $-$  & $-$ \\
\hline
$\ell_{24}$                      & $+$  & $-$  & $+$  & $-$  & $-$  & $-$  & $+$  & $+$  & $-$  & $-$ \\
\hline
$\ell_{25}$                      & $+$  & $+$  & $-$  & $-$  & $-$  & $+$  & $-$  & $+$  & $-$  & $-$ \\
\hline
$\ell_{26}$                      & $+$  & $+$  & $-$  & $-$  & $-$  & $+$  & $-$  & $+$  & $-$  &  $-$\\
\hline
$\ell_{27}$                      & $+$  & $-$  & $-$  & $+$  & $-$  & $+$  & $-$  & $-$  & $-$  & $+$ \\
\hline
$\ell_{28}$                      & $+$  & $-$  & $-$  & $+$  & $-$  & $+$  & $-$  & $-$  & $-$  & $+$ \\
\hline
$\ell_{29}$                      & $+$  & $-$  & $+$  & $-$  & $+$  & $-$  & $-$  & $-$  & $-$  & $+$ \\
\hline
$\ell_{30}$                      & $+$  & $-$  & $+$  & $-$  & $+$  & $-$  & $-$  & $-$  & $-$  & $+$ \\
\hline
\end{tabular}
\end{center}
\end{figure}

Now, let us describe the~intersection points of the~lines $\ell_{1},\ldots,\ell_{30}$. To do this, we set
\begin{align*}
\Sigma_{4}^1&=\big\{[1:0:0:0],[0:1:0:0],[0:0:1:0],[0:0:0:1]\big\},\\
\Sigma_{4}^2&=\big\{[1:1:1:-1],[1:1:-1:1],[1:-1:1:1],[-1:1:1:1]\big\},\\
\Sigma_{4}^3&=\big\{[1:1:1:1],[-1:-1:1:1],[1:-1:-1:1],[-1:1:-1:1]\big\},\\
\Sigma_{4}^4&=\big\{[0:0:1:1],[1:1:0:0],[0:0:-1:1],[1:-1:0:0]\big\},\\
\Sigma_{4}^5&=\big\{[1:0:1:0],[0:1:0:1],[-1:0:1:0],[0:-1:0:1]\big\},\\
\Sigma_{4}^6&=\big\{[0:1:1:0],[1:0:0:1],[0:-1:1:0],[-1:0:0:1]\big\},\\
\Sigma_{4}^7&=\big\{[i:0:0:1],[0:i:1:0],[-i:0:0:1], [0:-i:1:0]\big\},\\
\Sigma_{4}^8&=\big\{[i:0:1:0],[0:i:0:1],[0:-i:0:1],[-i:0:1:0]\big\},\\
\Sigma_{4}^9&=\big\{[i:1:0:0],[0:0:i:1],[-i:1:0:0],[0:0:-i:1]\big\},\\
\Sigma_{4}^{10}&=\big\{[i:i:1:1],[-i:-i:1:1],[i:-i:-1:1],[-i:i:-1:1]\big\},\\
\Sigma_{4}^{11}&=\big\{[1:i:i:1],[1:-i:-i:1],[-1:-i:i:1],[-1:i:-i:1]\big\},\\
\Sigma_{4}^{12}&=\big\{[1:i:-i:1],[-1:i:i:1],[-1:-i:-i:1],[1:-i:i:1]\big\},\\
\Sigma_{4}^{13}&=\big\{[i:1:i:1],[-i:1:-i:1],[-i:-1:i:1],[i:-1:-i:1]\big\},\\
\Sigma_{4}^{14}&=\big\{[i:1:-i:1],[i:-1:i:1],[-i:-1:-i:1],[-i:1:i:1]\big\},\\
\Sigma_{4}^{15}&=\big\{[i:i:-1:1],[-i:-i:-1:1],[i:-i:1:1],[-i:i:1:1]\big\}.
\end{align*}
Then the~subsets $\Sigma_{4}^{1},\ldots,\Sigma_{4}^{15}$ are $\mathbb{H}$-orbits of length $4$. Moreover, one has
$$
\Sigma_{4}^{1}\cup\Sigma_{4}^{2}\cup\cdots\cup\Sigma_{4}^{15}=\mathrm{Sing}\big(\ell_{1}+\ell_2+\cdots+\ell_{30}\big).
$$
So, for every $\ell_{i}$ and $\ell_{j}$ such that $\ell_{i}\ne\ell_{j}$ and $\ell_{i}\cap\ell_{j}\ne\varnothing$, one has $\ell_{i}\cap\ell_{j}\in\Sigma_{4}^{1}\cup\Sigma_{4}^{2}\cup\cdots\cup\Sigma_{4}^{15}$.
Furthermore, one can also check that
\begin{itemize}
\item every line among $\ell_{1},\ldots,\ell_{30}$ contains $6$ points in $\Sigma_{4}^{1}\cup\Sigma_{4}^{2}\cup\cdots\cup\Sigma_{4}^{15}$,
\item every point in $\Sigma_{4}^{1}\cup\Sigma_{4}^{2}\cup\cdots\cup\Sigma_{4}^{15}$ is contained in $3$ lines among $\ell_{1},\ldots,\ell_{30}$.
\end{itemize}

As in Remark~\ref{remark:S6}, let $\mathfrak{N}$ be the~normalizer of the~subgroup $\mathbb{H}$ in the~group $\mathrm{PGL}_4(\mathbb{C})$.
Then $\mathrm{Aut}(\mathbb{P}^3,\mathscr{C})\subset\mathfrak{N}$ and $\mathfrak{N}\cong\mathbb{H}.\mathfrak{S}_6$ by Remark~\ref{remark:S6}.
Moreover, one can show that
\begin{itemize}
\item the~group $\mathfrak{N}$ acts transitively on the~set $\{\mathcal{Q}_1,\ldots,\mathcal{Q}_{10}\}$,
\item the~group $\mathfrak{N}$ acts transitively on the~set $\{\ell_{1},\ldots,\ell_{30}\}$,
\item the~group $\mathfrak{N}$ acts transitively on the~set $\{\Sigma_{4}^1,\ldots,\Sigma_{4}^{15}\}$.
\end{itemize}

Now, we are ready to describe $\mathbb{H}$-orbits in $\mathbb{P}^3$. They can be described as follows:
\begin{enumerate}
\item $\Sigma_{4}^{1},\ldots,\Sigma_{4}^{15}$ are $\mathbb{H}$-orbits of length $4$;
\item $\mathbb{H}$-orbit of every point in $(\ell_{1}\cup\ell_2\cup\cdots\cup\ell_{30})\setminus(\Sigma_{4}^{1}\cup\Sigma_{4}^{2}\cup\cdots\cup\Sigma_{4}^{15})$ has length $8$;
\item $\mathbb{H}$-orbit of every point in $\mathbb{P}^3\setminus(\ell_{1}\cup\ell_2\cup\cdots\cup\ell_{30})$ has length $16$.
\end{enumerate}

\begin{lemma}
\label{lemma:30-lines-small-orbits}
The surface $\mathscr{S}$ does not contain $\mathbb{H}$-orbits of length $4$.
\end{lemma}

\begin{proof}
The assertion follows from \cite[Theorem~3]{Xiao}, because the~$\mathbb{H}$-action on the~minimal resolution of the~surface $\mathscr{S}$ is symplectic \cite{Hashimoto,Mukai}.
Instead, we can check this~\mbox{explicitly}.
Indeed, it is enough to check that $\mathscr{S}$ does not contain $\Sigma_{4}^1$, since the~group $\mathfrak{N}$ transitively
permutes the~orbits $\Sigma_{4}^{1},\ldots,\Sigma_{4}^{15}$.
If $\Sigma_{4}^1\subset \mathscr{S}$, then $\mathscr{S}$ is given by \eqref{equation:Kummer-quartic-surface} with $a=bcd=0$,
which implies that $\mathscr{S}$ has non-isolated singularities.
\end{proof}

\begin{corollary}
\label{corollary:30-lines-small-orbits}
Every line among $\ell_1,\ldots,\ell_{30}$ intersects $\mathscr{S}$ transversally by $4$ points.
\end{corollary}

\begin{proof}
Fix $k\in\{1,\ldots,15\}$.
If $|\ell_{2k-1}\cap \mathscr{S}|<4$, then the~subset~$(\ell_{2k-1}\cup\ell_{2k})\cap \mathscr{S}$ contains an $\mathbb{H}$-orbit of length $4$,
which contradicts Lemma~\ref{lemma:30-lines-small-orbits}.
Therefore, we have $|\ell_{2k-1}\cap \mathscr{S}|=4$.
Similarly, we see that $|\ell_{2k}\cap \mathscr{S}|=4$.
\end{proof}

Now, let us prove one result that plays a crucial role in the~proof of Theorem~\ref{theorem:main}.

\begin{lemma}
\label{lemma:curves}
Let $C$ be a possibly reducible $\mathbb{H}$-irreducible curve in $\mathbb{P}^3$ such that $\mathrm{deg}(C)<8$.
Then one of the~following two possibilities hold:
\begin{itemize}
\item[(a)] either $C=\ell_{2k-1}+\ell_{2k}$ for some $k\in\{1,\ldots,15\}$;
\item[(b)] or $C$ is a union of $4$ disjoint lines and $C\subset\mathcal{Q}_i$ for some $i\in\{1,\ldots,10\}$.
\end{itemize}
\end{lemma}

\begin{proof}
Intersecting $C$ with quadric surfaces $\mathcal{Q}_1,\ldots,\mathcal{Q}_{10}$, we conclude that $\mathrm{deg}(C)$ is even.
This gives $\mathrm{deg}(C)\in\{2,4,6\}$.

Suppose that $C$ is reducible. Since $|\mathbb{H}|=16$, we have the~following possibilities:
\begin{itemize}
\item[(i)] $C$ is a union of $2$ lines,
\item[(ii)] $C$ is a union of $4$ lines,
\item[(iii)] $C$ is a union of $2$ irreducible conics,
\item[(iv)] $C$ is a union of $3$ irreducible conics.
\item[(v)] $C$ is a union of $2$ irreducible plane cubics,
\item[(vi)] $C$ is a union of $2$ twisted cubics,
\end{itemize}
Since $\mathbb{P}^3$ does not have $\mathbb{H}$-orbits of length $2$ and $3$, cases (iii), (iv) and (v) are impossible.
Similarly, case (vi) is also impossible, because $\mumu_2^3$ cannot faithfully act on a rational curve.
Thus, either $C$ is a union of $2$ lines, or $C$ is a union of $4$ lines.

Suppose that $C=L_1+L_2$, where $L_1$ and $L_2$ are lines.
Then $\mathrm{Stab}_{\mathbb{H}}(L_1)\cong\mumu_2^3$, and this group cannot act faithfully on $L_1\cong\mathbb{P}^1$.
Therefore, there exists a non-trivial $g\in\mathrm{Stab}_{\mathbb{H}}(L_1)$ such that $g$ pointwise fixes the~line $L_1$.
But this means that $L_1$ is one of the~lines $\ell_1,\ldots,\ell_{30}$, so we have $C=\ell_{2k-1}+\ell_{2k}$ for some $k\in\{1,\ldots,15\}$ as required.

Suppose $C=L_1+L_2+L_3+L_4$, where $L_1$, $L_2$, $L_3$, $L_4$ are lines.
Then $\mathrm{Stab}_{\mathbb{H}}(L_1)\cong\mumu_2^2$.
Note that $\mathrm{Stab}_{\mathbb{H}}(L_1)$ must act faithfully on $L_1$,
because $L_1$ is not one of the~lines $\ell_1,\ldots,\ell_{30}$.
This implies that $L_1$ does not have $\mathrm{Stab}_{\mathbb{H}}(L_1)$-fixed points,
which implies that $\mathbb{P}^3$ also does not have $\mathrm{Stab}_{\mathbb{H}}(L_1)$-fixed points.
All subgroups in $\mathbb{H}$ isomorphic to $\mumu^2$ with these property are conjugated by the~action of the~group $\mathfrak{N}$.
Thus, we may assume that
$$
\mathrm{Stab}_{\mathbb{H}}\big(L_1\big)=\langle A_1A_2,A_3\rangle.
$$
This subgroup leaves invariant rulings of the~quadric surface $\mathcal{Q}_8\cong\mathbb{P}^1\times\mathbb{P}^1$.
To be precise, for every $[\lambda:\mu]\in\mathbb{P}^1$, the~group $\langle A_1A_2,A_3\rangle$ leaves invariant the~line
$$
\big\{\lambda x_0+\mu x_3=\lambda x_1+\mu x_2=0\big\}\subset\mathcal{Q}_8,
$$
and these are all  $\langle A_1A_2,A_3\rangle$-invariant lines in $\mathbb{P}^3$.
So, the~lines $L_1$, $L_2$, $L_3$, $L_4$ are~disjoint,
and all  of them are contained in the~quadric $\mathcal{Q}_8$.
Thus, we are done in this case.

Therefore, to complete the~proof of the~lemma, we may assume that  $C$ is irreducible.
Observe that the~curve $C$ is not planar, because $\mathbb{P}^3$ does not contain $\mathbb{H}$-invariant planes.
Moreover, the~curve $C$ is singular: otherwise its genus is $\leqslant 4$ by the~Castelnuovo bound,
but $\mathbb{H}$ cannot faithfully act on a smooth curve of genus less than $5$ by \cite[Lemma~3.2]{CheltsovSarikyan}.
Therefore, we conclude that $\mathrm{deg}(C)=6$, since otherwise the~curve $C$ would be planar.

We claim that the~curve $C$ does not contain $\mathbb{H}$-orbits of length $4$. Suppose that~it~does.
Since $\mathfrak{N}$ transitively permutes the~orbits $\Sigma_{4}^{1},\ldots,\Sigma_{4}^{15}$,
we may assume that $\Sigma_{4}^1\subset C$.
Then
$$
\Sigma_{4}^1\subset\mathrm{Sing}(C),
$$
because the~stabilizer in $\mathbb{H}$ of a smooth point in $C$ must be a cyclic group \cite[Lemma~2.7]{FZ}.
Let $\iota\colon\mathbb{P}^3\dasharrow\mathbb{P}^3$ be the~standard Cremona involution, which is given by
$$
[x_0:x_1:x_2:x_3]\mapsto [x_1x_2x_3:x_0x_2x_3:x_0x_1x_3:x_0x_1x_2].
$$
Then $\iota$ centralizes $\mathbb{H}$. On the~other hand, the~curve $\iota(C)$ is a conic,
because $\mathrm{deg}(C)=6$, and $C$ is singular at every point of the~$\mathbb{H}$-orbit $\Sigma_{4}^1$.
But $\mathbb{P}^3$ contains no $\mathbb{H}$-invariant conics, because it contains no $\mathbb{H}$-invariant planes.
Thus, $C$ contains no $\mathbb{H}$-orbits of length $4$.

Note that $\mathcal{Q}_1\cap\mathcal{Q}_2\cap\cdots\cap\mathcal{Q}_{10}=\varnothing$.
So, at least one quadric among $\mathcal{Q}_1,\ldots,\mathcal{Q}_{10}$ does not contain the~curve $C$.
Without loss of generality, we may assume that $C\not\subset\mathcal{Q}_1$. Then
$$
12=\mathcal{Q}_1\cdot C\geqslant|\mathcal{Q}_1\cap C|,
$$
which implies that the~intersection $\mathcal{Q}_1\cap C$ is an $\mathbb{H}$-orbit of length $8$,
because we already proved that $C$ does not contain $\mathbb{H}$-orbits of length~$4$.
For a point $P\in \mathcal{Q}_1\cap C$, we have
$$
12=\mathcal{Q}_1\cdot C=|\mathrm{Orb}_{\mathbb{H}}(P)|\big(\mathcal{Q}_1\cdot C\big)_P=8\big(\mathcal{Q}_1\cdot C\big)_P,
$$
which is impossible, since $12$ is not divisible by $8$.
\end{proof}

\begin{corollary}
\label{corollary:curves}
Let $\mathcal{Q}$ be an $\mathbb{H}$-invariant quadric among $\mathcal{Q}_1,\ldots,\mathcal{Q}_{10}$, and let $C=\mathscr{S}\vert_{\mathcal{Q}}$.
Then one of the following possibilities holds:
\begin{itemize}
\item $C$ is a smooth irreducible curve of degree $8$ and genus $9$;
\item $C$ is an $\mathbb{H}$-irreducible curve which is a union of two irreducible smooth quartic elliptic curves
that intersect each other transversally by an $\mathbb{H}$-orbit of length $8$;
\item $C=\mathcal{L}_4+\mathcal{L}_4^\prime$ for $\mathbb{H}$-irreducible curves $\mathcal{L}_4$ and $\mathcal{L}_4^\prime$
consisting of $4$ disjoint lines such that the~intersection $\mathcal{L}_4\cap\mathcal{L}_4^\prime$ is an $\mathbb{H}$-orbit of length $16$.
\end{itemize}
\end{corollary}

\begin{proof}
If $C$ is not $\mathbb{H}$-irreducible, then Lemma~\ref{lemma:curves} and Corollary~\ref{corollary:30-lines-small-orbits} imply the~assertion.
Thus, we may assume that $C$ is $\mathbb{H}$-irreducible.

Suppose that $C$ is reducible.
By Lemma~\ref{lemma:30-lines-small-orbits}, $C$ does not contain $\mathbb{H}$-orbits of length $4$.
Observe also that $\mathrm{Pic}^{\mathbb{H}}(\mathcal{Q})\cong\mathbb{Z}^2$.
Thus, since $C$ is $\mathbb{H}$-irreducible, we conclude that
\begin{itemize}
\item[(i)] either $C$ is a union of $4$ irreducible conics,
\item[(ii)] or $C$ is an $\mathbb{H}$-irreducible curve which is a union of two irreducible smooth quartic elliptic curves
that intersect each other transversally by an $\mathbb{H}$-orbit of length $8$.
\end{itemize}
In the latter case, we are done.
In the former case, we have $|\mathrm{Sing}(C)|\leqslant 12$, which implies that
$\mathrm{Sing}(C)$ is an~$\mathbb{H}$-orbit of length $8$, which immediately leads to a contradiction.

Thus, we may assume that $C$ is irreducible.
The arithmetic genus of the curve $C$ is $9$.
If $C$ is smooth, we are done.
If $C$ is singular, then the~genus of it normalization is $\leqslant 1$,
because $C$ does not contain $\mathbb{H}$-orbits of length $4$ by Lemma~\ref{lemma:30-lines-small-orbits}.
But $\mathbb{H}$ cannot faithfully act on a smooth curve of genus less than $5$ by \cite[Lemma~3.2]{CheltsovSarikyan}.
\end{proof}

Now, we are ready to prove Theorem~\ref{theorem:main}.

\begin{proof}[{Proof of Theorem~\ref{theorem:main}}]
Let $G$ be a subgroup in $\mathrm{Aut}(X)$ such that $\mathrm{Cl}^G(X)\cong\mathbb{Z}$ and
$$
\mathbb{H}\subseteq\upsilon(G),
$$
so $G$ contains a subgroup $\mathbb{H}^\rho$ for some homomorphism $\rho\colon\mathbb{H}\to\mumu_2$.
We must prove that the~threefold $X$ is $G$-birationally super-rigid. Suppose it is not $G$-birationally super-rigid.
Then there are a~positive integer $n$ and a $G$-invariant linear subsystem $\mathcal{M}\subset |nH|$
such that the~linear system $\mathcal{M}$ does not have fixed components, but $(X,\frac{2}{n}\mathcal{M})$ is not canonical.

Starting from this moment, we are going to forget about the~group $G$. In the~following, we will work only with its subgroup $\mathbb{H}^\rho$.
Recall that $\upsilon(\mathbb{H}^\rho)=\mathbb{H}$.

Let $Z$ be the~center of non-canonical singularities of the~log pair $(X,\frac{2}{n}\mathcal{M})$ that has maximal dimension.
We claim that $Z$ must be a point.
Indeed, suppose that $Z$ is a curve. Let~$M$ be sufficiently general surface in the~linear system $\mathcal{M}$.
Then
\begin{equation}
\label{equation:mult-big}
\mathrm{mult}_Z\big(M\big)>\frac{n}{2}
\end{equation}
by \cite[Theorem~4.5]{KoMo98}. Let us seek for a contradiction.

Let $\mathscr{Z}$ be an~$\mathbb{H}^\rho$-irreducible curve in $X$ whose irreducible components is the~curve $Z$,
Then, arguing as in the~proof of Proposition~\ref{proposition:80-49}, we see that
$$
H\cdot \mathscr{Z}\leqslant 7.
$$
In particular, we conclude that $\pi(\mathscr{Z})$ is a $\mathbb{H}$-invariant curve of degree $\leqslant 7$.
By Lemma~\ref{lemma:curves}, the~curve $\pi(Z)$ is a line, and one of the~following two possibilities hold:
\begin{itemize}
\item[(a)] either $\pi(\mathscr{Z})=\ell_{2k-1}+\ell_{2k}$ for some $k\in\{1,\ldots,15\}$;
\item[(b)] or $\pi(\mathscr{Z})$ is a union of $4$ disjoint lines and $\pi(\mathscr{Z})\subset\mathcal{Q}_i$ for some $i\in\{1,\ldots,10\}$.
\end{itemize}
Let us deal with these two cases separately.

Suppose we are in case (a). Without loss of generality, we may assume $\pi(\mathscr{Z})=\ell_{1}+\ell_{2}$.
Let $C_1$ and $C_2$ be the~preimages on the~threefold $X$ of the~lines $\ell_1$ and $\ell_2$, respectively.
Then it follows from Corollary~\ref{corollary:30-lines-small-orbits} that $C_1$ and $C_2$ are smooth irreducible  elliptic curves.
In particular, the~curves $C_1$ and $C_2$ are disjoint and
$$
\mathscr{Z}=C_{1}+C_{2}.
$$
Let $f\colon\widetilde{X}\to X$ be the~blow up of the~curves $C_1$ and $C_2$,
let $E_1$ and $E_2$ be the~$f$-exceptional surfaces such that $f(E_1)=C_1$ and $f(E_2)=C_2$,
and let $\widetilde{M}$ be the~proper transform on the~threefold $\widetilde{X}$ of the~surface $M$.
Then $|f^*(2H)-E_1-E_2|$ is base point free, so
\begin{multline*}
0\leqslant\big(f^*(2H)-E_1-E_2\big)^2\cdot\widetilde{M}=\\
=\big(f^*(2H)-E_1-E_2\big)^2\cdot\big(f^*(nH)-\mathrm{mult}_{Z}\big(M\big)(E_1+E_2)\big)=4n-8\mathrm{mult}_{Z}\big(M\big),
\end{multline*}
which contradicts \eqref{equation:mult-big}. This shows that case (a) is impossible.

Suppose we are in case (b). Without loss of generality, we may assume that $\pi(\mathscr{Z})\subset\mathcal{Q}_1$.
Let $S$ be the~preimage of the~quadric surface $\mathcal{Q}_1$ via the~double cover $\pi$.
Then it follows from Corollary~\ref{corollary:curves} that $S$ is an irreducible normal surface such that
\begin{itemize}
\item[(i)] either $S$ is a smooth K3 surface,
\item[(ii)] or $S$ is a singular K3 surface that has $8$ or $16$ ordinary double points.
\end{itemize}
Note that $\mathscr{Z}\subset S$ by construction.
Let $\mathcal{C}$ be the~preimage in $X$ of a sufficiently general line in the~quadric $\mathcal{Q}_1$
that intersect the~line $\pi(Z)$. Then $\mathcal{C}$ is a smooth irreducible elliptic curve,
which is contained in the~smooth locus of the~K3 surface $S$. Observe that $H\cdot \mathcal{C}=2$.
Moreover, we also have $|\mathcal{C}\cap\mathscr{Z}|\geqslant 4$.
Thus, since $\mathcal{C}\not\subset\mathrm{Supp}(M)$, we get
$$
2n=nH\cdot \mathcal{C}=M\cdot \mathcal{C}\geqslant\sum_{O\in\mathcal{C}\cap\mathscr{Z}}\mathrm{mult}_O(M)\geqslant\mathrm{mult}_Z\big(M\big)|\mathcal{C}\cap\mathscr{Z}|\geqslant 4\mathrm{mult}_Z\big(M\big),
$$
which contradicts \eqref{equation:mult-big}. This shows that case (b) is also impossible.

Hence, we see that $Z$ is a point. In particular, the~pair $(X,\frac{2}{n}\mathcal{M})$ is canonical away from finitely many points.
Now, arguing as in the~proof of Proposition~\ref{proposition:80-49}, we get
$$
Z\not\in\mathrm{Sing}(X).
$$

Let $M_1$ and $M_2$ be two general surfaces in $\mathcal{M}$. Using \cite{Pukhlikov} or \cite[Corollary 3.4]{Co00}, we get
\begin{equation}
\label{equation:Corti-2}
\big(M_1\cdot M_2\big)_Z>n^2.
\end{equation}

Let $P=\pi(Z)$. Then, arguing as in the~proof of Proposition~\ref{proposition:48-50}, we get $|\mathrm{Orb}_{\mathbb{H}}(P)|\ne 4$.

We claim that $|\mathrm{Orb}_{\mathbb{H}}(P)|\ne 8$.
Indeed, suppose $|\mathrm{Orb}_{\mathbb{H}}(P)|=8$.
Then
$$
P\in\ell_1\cup\ell_2\cup\cdots\cup\ell_{30}.
$$
Without loss of generality, we may assume that $P\in\ell_1$.
Let $C_1$ and $C_2$ be the~preimages on the~threefold $X$ of the~lines $\ell_1$ and $\ell_2$, respectively.
Recall that $C_1$ and $C_2$ are smooth irreducible elliptic curves,
and the~curve $C_1+C_2$ is $\mathbb{H}^\rho$-irreducible.
Write
$$
M_1\cdot M_2=m(C_1+C_2)+\Delta,
$$
where $m$ is a non-negative integer, and $\Delta$ is an effective one-cycle whose support does not contain the~curves $C_1$ and $C_2$.
Then $m\leqslant\frac{n^2}{2}$, because
$$
2n^2=H\cdot M_1\cdot M_2=mH\cdot(C_1+C_2)+H\cdot\Delta\leqslant mH\cdot(C_1+C_2)=4m.
$$
On the~other hand, since $C_1$ and $C_2$ are smooth curves, it follows from \eqref{equation:Corti-2} that
\begin{equation}
\label{equation:Corti-3}
\mathrm{mult}_O\big(\Delta\big)>n^2-m
\end{equation}
for every point $O\in\mathrm{Orb}_{\mathbb{H}^\rho}(Z)$. Note also that $Z\in C_1$ and $|\mathrm{Orb}_{\mathbb{H}^\rho}(Z)|\geqslant 8$.

Let $\mathcal{D}$ be the~linear subsystem in $|2H|$ that consists of surfaces passing through $C_1\cup C_2$.
Then, as we already implicitly mentioned, the~linear system $\mathcal{D}$ does not have base curves except for $C_1$ and $C_2$.
Therefore, if $D$ is a general surface in $\mathcal{D}$, then $D$ does not contain irreducible components of the~one-cycle $\Delta$,
so \eqref{equation:Corti-3} gives
\begin{multline*}
4n^2-8m=D\cdot\Delta\geqslant\sum_{O\in \mathrm{Orb}_{\mathbb{H}^\rho}(Z)}\mathrm{mult}_O\big(\Delta\big)=\\
=|\mathrm{Orb}_{\mathbb{H}^\rho}(Z)|\mathrm{mult}_Z\big(\Delta\big)>|\mathrm{Orb}_{\mathbb{H}^\rho}(Z)|(n^2-m)\geqslant 8(n^2-m),
\end{multline*}
which is absurd. This shows that $|\mathrm{Orb}_{\mathbb{H}}(P)|\ne 8$.

In particular, we see that $|\mathrm{Orb}_{\mathbb{H}^\rho}(Z)|=|\mathrm{Orb}_{\mathbb{H}}(P)|=16$ and $P\not\in\ell_1\cup\ell_2\cup\cdots\cup\ell_{30}$.

We claim that $P\not\in\mathcal{Q}_1\cup\mathcal{Q}_2\cup\cdots\cup\mathcal{Q}_{10}$.
Indeed, suppose that $P\in\mathcal{Q}_1\cup\mathcal{Q}_2\cup\cdots\cup\mathcal{Q}_{10}$.
Without loss of generality, we may assume that
$$
\pi(Z)=P\in\mathcal{Q}_1.
$$
As above, denote by $S$ the~preimage of the~quadric surface $\mathcal{Q}_1$ via the~double cover $\pi$.
Then $S$ is a K3 surface with at most ordinary double singularities,
and it follows from the~inversion of adjunction \cite[Theorem~5.50]{KoMo98}
that $(S,\frac{2}{n}\mathcal{M}\vert_{S})$ is not log canonical at $Z$.
Let $\lambda$ be the~largest rational number such that  $(S,\lambda\mathcal{M}\vert_{S})$ is log canonical at $Z$.
Then
$$
\mathrm{Orb}_{\mathbb{H}^\rho}(Z)\subseteq\mathrm{Nklt}\big(S,\lambda\mathcal{M}\vert_{S}\big).
$$
Note that the~locus $\mathrm{Nklt}(S,\lambda\mathcal{M}\vert_{S})$ is $\mathbb{H}^\rho$-invariant, because $\mathcal{M}$ and $S$ are $\mathbb{H}^\rho$-invariant.

Suppose $\mathrm{Nklt}(S,\lambda\mathcal{M}\vert_{S})$ contains an~$\mathbb{H}^\rho$-irreducible curve $C$ that passes through~$Z$.
This~means that $\lambda\mathcal{M}\vert_{S}=C+\Omega$,
where $\Omega$ is an effective $\mathbb{Q}$-linear system on $S$. Then
$$
H\cdot C\leqslant H\cdot(C+\Omega)=4n\lambda <8,
$$
hence $\pi(C)$ is a union of $4$ disjoint lines in $\mathcal{Q}_1$ by Lemma~\ref{lemma:curves}, since $P\not\in\ell_1\cup\ell_2\cup\cdots\cup\ell_{30}$.
Let~$\mathcal{C}$ be the~preimage in $X$ of a general line in  $\mathcal{Q}_1$ that intersect $\pi(C)$.
Then
$$
4\leqslant \mathcal{C}\cdot C\leqslant \mathcal{C}\cdot(C+\Omega)=\lambda n\big(H\cdot\mathcal{C}\big)=2\lambda n<4,
$$
which is absurd. So, the~locus $\mathrm{Nklt}(S,\lambda\mathcal{M}\vert_{S})$ contains no  curves that pass through $Z$.

Let $\mathcal{I}_S$ be the~multiplier ideal sheaf of the~pair $(S,\lambda\mathcal{M}\vert_{S})$,
let $\mathcal{L}_S$ be the~corresponding subscheme in $S$. Then
$$
\mathrm{Supp}\big(\mathcal{L}_S\big)=\mathrm{Nklt}\big(S,\lambda\mathcal{M}\vert_{S}\big).
$$
Now, applying Nadel's vanishing theorem \cite[Theorem~9.4.8]{Lazarsfeld}, we get
$$
h^1\big(S,\mathcal{I}_S\otimes\mathcal{O}_{S}(2H\vert_{S})\big)=0.
$$
Now, using the~Riemann--Roch theorem and Serre's vanishing, we obtain
$$
10=h^0\big(S,\mathcal{O}_{S}(2H\vert_{S})\big)\geqslant h^0\big(\mathcal{O}_{\mathcal{L}_S}\otimes\mathcal{O}_{S}(2H\vert_{S})\big)\geqslant|\mathrm{Orb}_{\mathbb{H}^\rho}(Z)|.
$$
because $\mathcal{L}_{S}$ has at least $|\mathrm{Orb}_{\mathbb{H}^\rho}(Z)|$
disjoint zero-dimensional components, whose supports are points in $\mathrm{Orb}_{\mathbb{H}^\rho}(Z)$,
because $\mathrm{Orb}_{\mathbb{H}^\rho}(Z)\subseteq\mathrm{Nklt}(S,\lambda\mathcal{M}\vert_{S})$,
and $\mathrm{Nklt}(S,\lambda\mathcal{M}\vert_{S})$ does not contain curves that are not disjoint from $\mathrm{Orb}_{\mathbb{H}^\rho}(Z)$.
Hence, we see that  $|\mathrm{Orb}_{\mathbb{H}^\rho}(Z)|\leqslant 10$,
which is impossible, since $|\mathrm{Orb}_{\mathbb{H}^\rho}(Z)|=16$.
This shows that
$$
\pi(Z)=P\not\in\mathcal{Q}_1\cup\mathcal{Q}_2\cup\cdots\cup\mathcal{Q}_{10}.
$$

Let us summarize what we proved so far. Recall that $\mathcal{M}$ is a mobile $\mathbb{H}^\rho$-invariant linear subsystem in $|nH|$,
the~log pair $(X,\frac{2}{n}\mathcal{M})$ is canonical away from finitely many points,
but the~singularities of the~pair $(X,\frac{2}{n}\mathcal{M})$ are not canonical at the~point $Z\in X$ such that
\begin{itemize}
\item $Z\not\in\mathrm{Sing}(X)$,
\item $\pi(Z)\not\in\ell_1\cup\ell_2\cup\cdots\cup\ell_{30}$,
\item $\pi(Z)\not\in\mathcal{Q}_1\cup\mathcal{Q}_2\cup\cdots\cup\mathcal{Q}_{10}$,
\item $|\mathrm{Orb}_{\mathbb{H}^\rho}(Z)|=16$,
\item $|\mathrm{Orb}_{\mathbb{H}}(\pi(Z))|=16$.
\end{itemize}
By Lemma~\ref{lemma:curves}, $\pi(Z)$ is not contained in any $\mathbb{H}$-invariant curve whose degree is at most~$7$.
Let us use this and Nadel's vanishing \cite[Theorem~9.4.8]{Lazarsfeld} to derive a contradiction.

As in the~proofs of Propositions~\ref{proposition:80-49} and \ref{proposition:48-50},
we observe that  $(X,\frac{3}{n}\mathcal{M})$ is not log canonical at the~point $Z$, because $X$ is smooth at $Z$.
Let $\mu$ be the~largest rational number such~that the~log pair $(X,\mu\mathcal{M})$ is log canonical at $Z$.
Then $\mu<\frac{3}{n}$ and
$$
\mathrm{Orb}_{\mathbb{H}^\rho}(Z)\subseteq\mathrm{Nklt}\big(X,\mu\mathcal{M}\big).
$$
Moreover, if the~locus $\mathrm{Nklt}(X,\mu\mathcal{M})$ contains an $\mathbb{H}^\rho$-irreducible curve $C$,
then arguing as in the~proof of Proposition~\ref{proposition:80-49}, we see that
$$
\mathrm{deg}\big(\pi(C)\big)\leqslant H\cdot C\leqslant 4,
$$
which implies that the~curve $C$ does not pass through $Z$.
Hence, we conclude  that~every point of the~orbit $\mathrm{Orb}_{\mathbb{H}^\rho}(Z)$ is an isolated irreducible component of the~locus $\mathrm{Nklt}(X,\mu\mathcal{M})$.

Let $\mathcal{I}$ be the~multiplier ideal sheaf of the~pair $(X,\mu\mathcal{M})$,
and let $\mathcal{L}$ be the~corresponding subscheme in $X$. Then
$$
\mathrm{Supp}\big(\mathcal{L}\big)=\mathrm{Nklt}\big(X,\mu\mathcal{M}\big),
$$
so the~subscheme $\mathcal{L}$ contains at least $|\mathrm{Orb}_{\mathbb{H}^\rho}(Z)|=16$ zero-dimensional components
whose supports are points in the~orbit $\mathrm{Orb}_{\mathbb{H}^\rho}(Z)$.
On the~other hand, we have
$$
h^1\big(X,\mathcal{I}\otimes\mathcal{O}_{X}(H)\big)=0
$$
by Nadel's vanishing theorem \cite[Theorem~9.4.8]{Lazarsfeld}. This gives
$$
4=h^0\big(X,\mathcal{O}_{X}(H)\big)\geqslant h^0\big(\mathcal{O}_{\mathcal{L}}\otimes\mathcal{O}_{X}(H)\big)\geqslant |\mathrm{Orb}_{\mathbb{H}^\rho}(Z)|=16,
$$
which is absurd. The obtained contradiction completes the~proof of Theorem~\ref{theorem:main}.
\end{proof}

Let us conclude this paper with one~application of Theorem~\ref{theorem:main},
which was the~initial motivation for this paper --- we were looking for various embeddings $\mumu_2^4\rtimes\mumu_3\hookrightarrow\mathrm{Bir}(\mathbb{P}^3)$.

\begin{example}[{cf. Examples~\ref{example:recover-curve-from-quartic}, \ref{example:48-50}, \ref{example:48-50-class-group}}]
\label{example:48-50-final}
Let $G_{48,50}$ be the~subgroup in $\mathrm{PGL}_4(\mathbb{C})$ generated~by
\begin{multline*}
A_1=\begin{pmatrix}
-1 & 0 & 0 & 0\\
0 & 1 & 0 & 0\\
0 & 0 & -1 & 0\\
0 & 0 & 0 & 1
\end{pmatrix},
A_2=\begin{pmatrix}
-1 & 0 & 0 & 0\\
0 & -1 & 0 & 0\\
0 & 0 & 1 & 0\\
0 & 0 & 0 & 1
\end{pmatrix},\\
A_3=\begin{pmatrix}
0 & 1 & 0 & 0\\
1 & 0 & 0 & 0\\
0 & 0 & 0 & 1\\
0 & 0 & 1 & 0
\end{pmatrix},
A_4=\begin{pmatrix}
0 & 0 & 0 & 1\\
0 & 0 & 1 & 0\\
0 & 1 & 0 & 0\\
1 & 0 & 0 & 0
\end{pmatrix},
A_5=\begin{pmatrix}
0 & 0 & 1 & 0\\
1 & 0 & 0 & 0\\
0 & 1 & 0 & 0\\
0 & 0 & 0 & 1
\end{pmatrix}.
\end{multline*}
Then one can check that $G_{48,50}\cong\mumu_2^4\rtimes\mumu_3$ and the~GAP ID of the~group $G_{48,50}$ is [48,50].
For~every $t\in\mathbb{C}\setminus\{\pm 1,\pm\sqrt{3}i\}$,
let $S_t$ be the~quartic surface in $\mathbb{P}^3$   given by the~equation~\eqref{equation:Kummer-quartic-surface-48-50},
i.e.~the surface $S_t$ is the~quartic surface in $\mathbb{P}^3$ given by
$$
x_0^4+x_1^4+x_2^4+x_3^4-(t^2+1)(x_0^2x_1^2+x_2^2x_3^2+x_0^2x_2^2+x_1^2x_3^2+x_0^2x_3^2+x_1^2x_2^2)+2(t^3+3t)x_0x_1x_2x_3=0.
$$
Then $S_t$ is $G_{48,50}$-invariant, and $S_t$ has $16$ ordinary double singularities (see Example~\ref{example:recover-curve-from-quartic}).
Now, let $X_t$ be the~hypersurface in $\mathbb{P}(1,1,1,1,2)$ that is given by
$$
w^2=x_0^4+x_1^4+x_2^4+x_3^4-(t^2+1)(x_0^2x_1^2+x_2^2x_3^2+x_0^2x_2^2+x_1^2x_3^2+x_0^2x_3^2+x_1^2x_2^2)+2(t^3+3t)x_0x_1x_2x_3,
$$
where we consider $x_0$, $x_1$, $x_2$, $x_3$ as homogeneous coordinates on $\mathbb{P}(1,1,1,1,2)$ of weight~$1$,
and $w$ is a coordinate of weight $2$.
Consider the~faithful action $G_{48,50}\curvearrowright X_t$ given~by
\begin{align*}
A_1\colon [x_0:x_1:x_2:x_3:w]&\mapsto[-x_0:x_1:-x_2:x_3:w],\\
A_2\colon [x_0:x_1:x_2:x_3:w]&\mapsto[-x_0:x_1:-x_2:x_3:w],\\
A_3\colon [x_0:x_1:x_2:x_3:w]&\mapsto[x_1:x_2:x_3:x_2:w],\\
A_4\colon [x_0:x_1:x_2:x_3:w]&\mapsto[x_3:x_2:x_1:x_0:w],\\
A_5\colon [x_0:x_1:x_2:x_3:w]&\mapsto[x_1:x_2:x_0:x_3:w].
\end{align*}
Since the~threefold $X_t$ is $G_{48,50}$-invariant, this gives an embedding $G_{48,50}\hookrightarrow\mathrm{Aut}(X_t)$.
Then it follows from Theorem~\ref{theorem:main} that the~threefold $X_t$ is $G_{48,50}$-birationally super-rigid.
In particular, for any $t_1\ne t_2$ in $\mathbb{C}\setminus\{\pm 1,\pm\sqrt{3}i\}$,
the following conditions are equivalent:
\begin{itemize}
\item the~threefolds $X_{t_1}$ and $X_{t_2}$ are $G_{48,50}$-birational;
\item the~surfaces $S_{t_1}$ and $S_{t_2}$ are projectively equivalent.
\end{itemize}
Recall that $X_t$ is rational.
For $t\in\mathbb{C}\setminus\{\pm 1,\pm\sqrt{3}i\}$, fix a birational map $\chi_t\colon\mathbb{P}^3\dasharrow X_t$,
and consider the~monomorphism $\eta_t\colon G_{48,50}\hookrightarrow\mathrm{Bir}(\mathbb{P}^3)$ that is given by $g\mapsto \chi_t^{-1}\circ g\circ\chi_t$.
Then, for any $t_1\ne t_2$ in $\mathbb{C}\setminus\{\pm 1,\pm\sqrt{3}i\}$, we have the~following assertion:
$$\hspace*{-0.1cm}
\eta_{t_1}(G_{48,50})\ \text{and}\ \eta_{t_1}(G_{48,50})\ \text{are conjugate in} \ \mathrm{Bir}(\mathbb{P}^3) \iff X_{t_1}\ \text{and}\ X_{t_2}\ \text{are $G_{48,50}$-birational}.
$$
Thus, if $t_1\ne t_2$ are general, then $\eta_{t_1}(G_{48,50})$ and $\eta_{t_1}(G_{48,50})$ are not conjugate in~$\mathrm{Bir}(\mathbb{P}^3)$.
Similarly, we see that $\eta_t(G_{48,50})$ is not conjugate in~$\mathrm{Bir}(\mathbb{P}^3)$ to the~group  $G_{48,50}\subset\mathrm{PGL}_4(\mathbb{C})$,
which also follows from \cite{CheltsovSarikyan}.
Can we show this using other obstructions \cite{BogomolovProkhorov,KontsevichPestunTschinkel,HassettKreschTschinkel}?
\end{example}

\medskip

\textbf{Acknowledgments.}
We would like to thank Michela Artebani,
Arnaud Beauville, Fabrizio Catanese, Ciro Ciliberto, Igor Dolgachev, Alexander Kuznetsov, Yuri Prokhorov, Eduardo Luis Sola Conde, Andrei Trepalin, 	
Alessandro Verra for very helpful comments.
We are grateful to Artem Avilov for spotting a gap in the  proof of Corollary~\ref{corollary:curves}.

\end{document}